\documentclass[11pt]{article}
\usepackage{amsmath,amsfonts,amsthm,amssymb,mathrsfs,mathtools,ulem}
\usepackage{graphicx}
\usepackage[usenames,dvipsnames]{color}
\usepackage{float}
\usepackage{graphicx}
\usepackage{subfigure}
\usepackage{caption}
\usepackage{cite}
\usepackage{url}
\usepackage{caption}
\usepackage{epstopdf}
\usepackage{hyperref}

\usepackage{color}

\newcommand{\dint}{\displaystyle\int}

\newcommand\redout{\bgroup\markoverwith
{\textcolor{red}{\rule[0.5ex]{2pt}{0.8pt}}}\ULon}

\theoremstyle{plain}
\newtheorem{theorem}{Theorem}[section]
\newtheorem{hyp}{Hypothesis}[section]
\newtheorem{corollary}[theorem]{Corollary}
\newtheorem{lemma}[theorem]{Lemma}
\newtheorem{proposition}[theorem]{Proposition}
\newtheorem{example}{Examples}[section]

\theoremstyle{definition}

\newtheorem{remark}[theorem]{Remark}

\numberwithin{equation}{section}
\numberwithin{theorem}{section}

\usepackage{a4wide}
\usepackage{geometry}
\title{Hausdorff dimensions and Hitting probabilities for some general Gaussian processes}
\author{ }
\date{}

\begin{document}
\maketitle
 
 \begin{center}
     \author{FREDERI VIENS\footnote{Partially supported by the National Science Foundation award number DMS 1811779.}\\
     Department of Statistics and Probability,\\
     Michigan State University, USA\\
     e-mail\textup{: \texttt{viens@msu.edu}}}
 \end{center}
 \begin{center}
     
   \author{ MOHAMED ERRAOUI\\
   Department of mathematics, Faculty of science El jadida,\\ Chouaïb Doukkali University,
   Morocco \\
  e-mail\textup{: \texttt{erraoui@uca.ac.ma}}}

\begin{center}
         \author{ YOUSSEF HAKIKI\footnote{Supported by National Center for Scientific and Technological Research (CNRST)}\\
   Department of mathematics, Faculty of science Semlalia,\\ Cadi Ayyad University, 2390 Marrakesh, Morocco \\
  e-mail\textup{: \texttt{youssef.hakiki@ced.uca.ma}}}
 
 \end{center} 
\end{center}


\begin{abstract}
\small Let $B$ be a $d$-dimensional Gaussian process on $\mathbb{R}$, where the component are independents copies of a scalar Gaussian process $B_0$ on $\mathbb{R}_+$ with a given general variance function $\gamma^2(r)=\operatorname{Var}\left(B_0(r)\right)$ and a canonical metric $\delta(t,s):=(\mathbb{E}\left(B_0(t)-B_0(s)\right)^2)^{1/2}$ which is commensurate with $\gamma(t-s)$. We provide some general condition on $\gamma$ so that for any Borel set $E\subset [0,1]$, the Hausdorff dimension of the image $B(E)$ is constant a.s., and we explicit this constant. Also, we derive under some mild assumptions on $\gamma\,$ an upper and lower bounds of $\mathbb{P}\left\{B(E)\cap F\neq \emptyset \right\}$ in terms of the corresponding Hausdorff measure and capacity of $E\times F$. Some upper and lower bounds for the essential supremum norm of the Hausdorff dimension of $B(E)\cap F$ and $E\cap B^{-1}(F)$ are also given in terms of $d$ and the corresponding Hausdorff dimensions of $E\times F$, $E$, and $F$.   
\end{abstract}

\textbf{Keywords:} Gaussian process, Hitting probabilities, 
Hausdorff dimension, Capacity.
\vspace{0,2cm}

\textbf{Mathematics Subject Classification} \quad 60J45, 60G17, 28A78, 60G15
\section{Introduction}
This paper aims to study some fractal properties for Gaussian processes that have a general covariance structure, such as the Hausdorff dimension of the image set, the hitting probabilities problem, and the Hausdorff dimension of the level sets and the pre-images. The motivation arises from the high irregularity presented by certain types of Gaussian processes, see, for example, the family of processes $B^{\gamma}$ defined in \cite{Viens&Mocioalca2005}, which associate to each appropriate function $\gamma$ its corresponding process $B^{\gamma}$, which is defined through the following Volterra representation 
\begin{align}\label{volt repr}
    B^{\gamma}(t):=\dint_0^{t}\sqrt{\left(\frac{d\gamma^2}{dt}\right)(t-s)}dW(s),
\end{align}
where $W$ is a standard Brownian motion. In the particular case $\gamma(r):=\log^{-\beta}(1/r)$, where $\beta>1/2$, the process $B^{\gamma}$ is an element of a class of Gaussian processes called the class of logarithmic Brownian motions, which will be considered as the highly irregular class of continuous Gaussian processes for our study (there are some other highly irregular classes of Gaussian processes, but they are not continuous). In that case, considering the well known fact that the process $B^{\gamma}$ should have the function $h:r\mapsto \gamma(r)\log^{1/2}(1/r)$ as uniform modulus of continuity, up to a deterministic constant, we may deduce that $B^{\gamma}$ is no longer Hölder continuous for any order $\alpha\in (0,1)$, which illustrate a high level of irregularity. So, most of the results in the literature about the fractal properties for Gaussian processes, for the hitting probabilities problem see for example \cite{Xiao2009,Bierme&Xiao,Chen&Xiao}, and for the Hausdorff dimension of the image and the graph sets see \cite{Hawkes}, do not apply to the case of logarithmic Brownian motion. Because the conditions assumed in those previous works restrict the processes studied to the Hölder continuity scale type, i.e. when  $\gamma\left(r\right)\lesssim r^{\alpha}$ for some $\alpha\in (0,1)$. Since there are many regularity scales between the Hölder continuity scale and the logaritmic scale mentioned above, we need some precise quantitative results on the fractals properties for a class of Gaussian  processes $B$ that are satisfying only the condition \eqref{commensurate}, those results hold only under some general conditions on the variance function $\gamma$ that are already satisfied by large class of processes within and/or beyond the Hölder scale. 

We note that when $\gamma^2$ is of class $\mathcal{C}^2$ on $(0,\infty)$, $\lim_0\gamma=0$, and $\gamma^2$ is increasing and concave, the Gaussian processes $B^{\gamma}$ defined above in \eqref{volt repr} satisfiy $\eqref{commensurate}$ with $l=2$ (see Proposition 1 in \cite{Viens&Mocioalca2005}). This model of concentrating only on the commensurability condition \eqref{commensurate} has the power to relax the restriction of stationarity of increments (see Proposition 5 in \cite{Viens&Mocioalca2005}), and the Hölder continuity, as illustrated above by the logarithmic Brownian motion. Another interesting class of Gaussian processes with non-stationary increments, which satisfy \eqref{commensurate}, are the solutions of the linear stochastic heat equation, see those studied in \cite{Ouahhabi&Tudor2013}.

In section 2, we give some general hypotheses on $\gamma$, that are important to ensure some interesting properties for process $B$, as the two-points local nondeterminism in \eqref{two pts LND}, and the estimation \eqref{estim small ball 1}, in order to provide the optimal upper and lower bounds for both of the Hausdorff dimension of the image $B(E)$ and the hitting probabilities in the sections 3 and 4, respectively. 

The objective of section 3 in this paper is to find when the variance function $\gamma$ is strictly concave in a neighborhood of $0$, an explicit formula for the Hausdorff dimension of the image $B(E)$ where $E\subset \mathbb{R}_+$ is a Borel set. Hawkes resolved this problem at first in his paper \cite{Hawkes}, but only in the case of stationary increments and under the strong condition $ind(\gamma)>0$, where $ind(\cdot)$ is the lower index of $\gamma$, it will be defined in \eqref{def index}. Our aim in this section is to relax those two last conditions. Precisely, we will give a lower bound for the Hausdorff dimension of the image $\dim_{Euc}\left(B(E)\right)$, and under the condition \eqref{nice condition}, we show that the random variable $\dim_{Euc}\left(B(E)\right)$ is almost surely constant by proving that the lower bound is also an upper bound. This constant is expressed as the minimum between $d$ and $\dim_{\delta}(E)$ where $\dim_{\delta}(\cdot)$ denote the Hausdorff dimension associated to the canonical metric $\delta$. We note that Lemma \ref{lem check nice cond} may illustrates that \eqref{nice condition} is general than the condition of Hawkes "$ind(\gamma)>0$", and is even satisfied by an important class of functions zero lower index as we will see bellow.

In section 4, our investigation will be focused on providing upper and lower bounds for the probabilities of the event $\left\{B(E)\cap F\neq \varnothing \right\}$ where $E\subset \mathbb{R}_+$ and $F\subset \mathbb{R}^d$ are Borel sets, in terms of $\mathcal{H}_{\rho_{\delta}}^d\left(E\times F\right)$ and $\mathcal{C}_{\rho_{\delta},d}\left(E\times F\right)$, respectively. $\mathcal{H}_{\rho_{\delta}}^d\left(\cdot\right)$ and $\mathcal{C}_{\rho_{\delta},d}\left(\cdot\right)$ denote $d$-dimensional Hausdorff measure and the Bessel-Riesz type capacity of order $d$ with respect to an appropriate metric $\rho_{\delta}$ depending on the canonical metric $\delta$, which will be defined in the sequel. Similar to section 3, the lower bound is given just under the condition of strict concavity near 0 of $\gamma$. However, the upper bound can be obtained with the help of Lemmas \ref{lem upper bound} and \ref{cor estim small ball} under the mild condition \eqref{condition raisonable} on $\gamma$, which is stronger than \eqref{nice condition}, but it remains satisfied by almost all examples of interest. A less optimal upper bound for the probability of the above event is also given under the weaker condition \eqref{nice condition} in terms of $\mathcal{H}_{\rho_{\delta}}^{d-\varepsilon}\left(E\times F\right)$ for any $\varepsilon>0$ small enough.       

Our motivation for section 5 is the following: when the random intersections $B(E)\cap F$ and $B^{-1}(F)\cap E$ are non-empty with positive probability, it is natural to ask about their Hausdorff dimensions. In general, those Hausdorff dimensions are not necessarily a.s. constant, like it was illustrated by Khoshnevisan and Xiao  through an example -due to Gregory Lawler- in the introduction of their paper \cite{KX2015}. For this reason we seek to give some upper and lower bounds of the $L^{\infty}\left(\mathbb{P}\right)$-norm of those Hausdorff dimensions. We note that when $B$ is a standard Brownian motion, Khoshnivisan and Xiao have obtained in \cite{KX2015} an explicit formula for the essential supremum of the Hausdorff dimension of  $E\cap B^{-1}(F)$ and $B(E)\cap F$ in terms of $d$ and $\dim_{\rho_{\delta}}\left(E\times F\right)$. 
A generalization of their result to the fractional Brownian case was proven by Erraoui and Hakiki in \cite{Erraoui&Hakiki 3}, by giving only an upper and lower bound. Our goal is to generalize those results to the Gaussian processes satisfying \eqref{commensurate} and under general conditions on $\gamma$.


\section{Preliminaries}
 Let $ \{B_0(t),t\in \mathbb{R}_+\}$ be a real valued centered continuous Gaussian process defined on a complete  probability space $(\Omega,\mathcal{F},\mathbb{P})$, defining the canonical metric $\delta$ of $B$ on $(\mathbb{R}_+)^2$ by $$ \delta(s,t):=\left(\mathbb{E}(B_0(s)-B_0(t))^2\right)^{1/2}.$$
Let $\gamma$ be continuous increasing function on $\mathbb{R}_+$ (or possibly only on a neighborhood of $0$ in $\mathbb{R}_+$), such that $ \lim_0 \gamma =0$ and for some constant $l\geq 1$ we have, for all $s,t \in \mathbb{R}_+$ 
\begin{align}
\left\{\begin{array}{rcl}
    &\mathbb{E}\left(B_0(t)\right)^2=\gamma^2(t) \\
    \\
   & and \\
   \\
     &1/\sqrt{l}\,\gamma\left(|t-s|\right)\leq \delta(t,s)\leq \sqrt{l}\,\gamma(|t-s|).    
\end{array}\right.\label{commensurate}
\end{align}
Now, we consider the $\mathbb{R}^d$-valued process $B=\{B(t): t\in \mathbb{R}_+\}$ defined by 
\begin{align}\label{d iid copies}
    B(t)=(B_1(t),...,B_d(t)),\quad t\in \mathbb{R}_+,
\end{align}
where $B_1,...,B_d$ are independent copies of $B_0$. Let us consider the following hypotheses 
 
 \begin{hyp}\label{H1} The increasing function $\gamma$ is concave in a neighborhood of the origin, and for all $0<a<b<\infty$, there exists $\varepsilon>0$ such that $\gamma^{\prime}(\varepsilon^+)>\sqrt{l}\, \gamma^{\prime}(a^{-})$.
 \end{hyp}
 \begin{hyp} \label{Hyp2} For all $0<a<b<\infty$, there exists $\varepsilon>0$ and $c_0\in (0,1/\sqrt{l})$, such that for all $s,t\in [a,b]$ with $0<t-s\leq \varepsilon$,
 $$ \gamma(t)-\gamma(s)\leq c_0\gamma(t-s).$$
 \end{hyp}
 It was proven in Lemma 2.3 in \cite{Eulalia&Viens2013} that Hypothesis \ref{H1} implies Hypothesis \ref{Hyp2}, and that under the strong condition $\gamma^{\prime}(0^{+})=\infty$, the constant $c_0$ in Hypothesis \ref{Hyp2} can be chosen arbitrary small. The following lemma  establish, under Hypothesis \ref{Hyp2}, the so called two point local-nondeterminism. It was proven in Lemma 2.4 of \cite{Eulalia&Viens2013}.
\begin{lemma}\label{lem two pts LND}
Assume Hypothesis \ref{Hyp2}. Then for all $0<a<b<\infty$, there exist $\varepsilon>0$ and positive constant $c_1$ depending only on $a,b$, such that for all $s,t\in [a,b]$ with $|t-s|\leq \varepsilon$,

\begin{align}\label{two pts LND}
    Var\left(B_0(t)/B_0(s)\right)\geq c_1\, \gamma^2(|t-s|).
\end{align}
\end{lemma}
We denote by $B_{\delta}(t,r)=\{s\in \mathbb{R}_+: \delta(s,t)\leq r\}$ the closed ball of center $t$ and radius $r$. 
The following lemma is useful for the proof of the upper bounds for the Hausdorff dimension in \ref{Hausd dim image}. It can be seen also as an improvement of both of Proposition 3.1. and Proposition 4.1. in \cite{Eulalia&Viens2013}. The proof that we give here rely on the arguments of the classical Gaussian path regularity theory, like it was done in Lemma 3.1 in \cite{Bierme&Xiao} and Lemma 7.8 in \cite{Xiao2009}.
\begin{lemma}\label{lem upper bound}
Let $0<a<b<\infty$, and $I:=[a,b]$.
Then for all $M>0$, there exist a positive constants $c_2$ and $r_0$ such that for all $r\in (0,r_0)$, $t\in I$  and $z\in [-M,M]^d$ we have 
\begin{equation}
		\mathbb{P}\left\{\inf _{ s\in B_{\delta}(t,r)\cap I}\|B(s)-z\| \leqslant r\right\} \leqslant c_2 (r+f_{\gamma}(r))^{d},\label{estim small ball 1}
		\end{equation}
		where $\|\cdot\|$ is the Euclidean metric, and $f_{\gamma}$ is defined by $$f_{\gamma}(r):= r\,\sqrt{\log 2}+\int_0^{1/2}\gamma\left(\gamma^{-1}(r\sqrt{l})\,y\right)\frac{dy}{y \sqrt{\log(1/y)}}.$$ 
\end{lemma}
\begin{corollary}\label{cor estim small ball}
Assume that there exists $x_0,c_3$ such that 
\begin{align}
    \int_{0}^{1 / 2} \gamma(x y) \frac{d y}{y \sqrt{\log (1 / y)}} \leq c_3 \gamma(x)\label{condition raisonable}
\end{align}
for all $x\in [0,x_0]$. Then, there is some constant $c_4$ depending on $\gamma$, $I$, $r_0$, and $x_0$, such that for all $z\in [-M,M]^d$ and for all $r\in (0,r_0\wedge \gamma(x_0))$ we have 
\begin{align}
\mathbb{P}\left\{\inf _{ s\in B_{\delta}(t,r)\cap I}\|B(s)-z\| \leqslant r\right\} \leqslant c_4 r^{d}.\label{estim small ball 2}
\end{align}
\end{corollary}
    It is easy to check that any power function $\gamma(x)=x^{H}$ with $H\in (0,1)$ satisfies \eqref{condition raisonable}. Moreover, we will show that \eqref{condition raisonable} is satisfied by the class of regularly varying functions, which include all power functions. First of all, a function $\gamma$ is said to be regularly function with index $0<\alpha<1$ if it can be written as $$ 
    \gamma(x)=x^{\alpha}\, L(x),
    $$
    where $L(x): [0,x_0)\rightarrow [0,\infty)$ is a slowly varying function near $0$ in the sense of Karamata, then it can be represented by 
    \begin{align}
        L(x)=\exp\left(\eta(x) + \int_{x}^{A}\frac{\varepsilon(t)}{t}dt\right),
    \end{align}
    where $\eta: [0,x_0)\rightarrow \mathbb{R}$, $\varepsilon: [0,A)\rightarrow \mathbb{R}$ are Borel measurable and bounded functions, and there exists a finite constant $c_0$ such that  
    $$ 
    \lim_{x\rightarrow 0}\eta(x)=c_0, \quad \text{ and  } \quad \lim_{x\rightarrow 0}\varepsilon(x)=0.
    $$
    For more properties of regularly varying functions see Seneta \cite{Seneta} or Bingham et al. \cite{Bingham et al}.
    
\begin{proposition}\label{prop RVF}
Let $\gamma$ be a regularly varying function near 0, with index $0<\alpha<1$. Then $\gamma$ satisfies \eqref{condition raisonable}.
\end{proposition}    

\begin{proof}
First, using the representation $\gamma(x)=x^{\alpha}\,L(x)$ for all $x\in (0,x_0)$, and thanks to the result of Adamović; see Proposition 1.3.4 in \cite{Bingham et al}, we may assume without loss of generality that the slowly varying part $L(\cdot)$ is $\mathcal{C}^{\infty}$. Since the constant $c_0$ is finite and we are interested only to the asymptotic behavior of $\gamma$ near 0, we restrict our attention to the case where the slowly varying part is given by 
\begin{align}
    L(x)=M_0\, \exp\left(\int_x^A\frac{\varepsilon(t)}{t}dt\right),
\end{align}
where $M_0>0$. Now, we check the condition \eqref{condition raisonable}; 
\begin{align}
    \begin{aligned}
        \dint_0^{1/2}\gamma(xy)\frac{dy}{y\sqrt{\log(1/y)}}&=x^{\alpha}\,\dint_0^{1/2}L(xy)\frac{dy}{y^{1-\alpha}\sqrt{\log(1/y)}}\\
    &\leq \frac{x^{\alpha}}{\log^{1/2}(2)}\, \dint_0^{1/2}L(xy)\frac{dy}{y^{1-\alpha}}\\
    &\leq c_1\, \dint_0^xL(z)\frac{dz}{z^{1-\alpha}}.
    \end{aligned}
\end{align}
Then, it suffice to show that $\limsup_{x\downarrow 0}\frac{\int_0^xL(z)z^{\alpha-1}dz}{\gamma(x)}<\infty$. It is easy to check that $\gamma^{\prime}(x)=x^{\alpha-1}\,L(x)\left(\alpha-\varepsilon(x)\right)$. Thus we may apply the Hospital rule to get that 
$$
\lim_{x\downarrow 0}\frac{\int_0^xL(z)z^{\alpha-1}dz}{\gamma(x)}=\lim_{\downarrow 0}\frac{x^{\alpha-1}\,L(x)}{\gamma^{\prime}(x)}=1/\alpha<\infty,
$$
which finishes the proof.
\end{proof}

\begin{example}\label{exple of RVF}
Here is some families of regularly varying functions that are immediately satisfying \eqref{condition raisonable} due to Proposition \ref{prop RVF}
\begin{itemize}
    \item[i)]$\gamma_{\alpha,\beta}(r):=r^{\alpha}\log^{\beta}(1/r)$ for $\beta \in \mathbb{R}$ and $\alpha\in (0,1)$,
    \item[ii)]$\gamma_{\alpha,\beta}(x):=x^{\alpha}\,\exp\left(\log^{\beta}(1/x)\right)$ for $\beta \in (0,1)$ and $\alpha \in (0,1)$,
    \item[iii)]$\gamma_{\alpha}(x):=x^{\alpha}\, \exp\left(\frac{\log(1/x)}{\log\left(\log(1/x)\right)}\right)$ for $\alpha\in (0,1)$.
\end{itemize}
\end{example}

\begin{remark}\label{ex satisfy reas cond}
 Notice that condition \eqref{condition raisonable} is also satisfied by the class of all "gauge" functions considered by Sanz-solé and Calleja in \cite{Marta&Calleja2021} 
 . Indeed, they considered a "gauge" function $\gamma(\cdot)$ which satisfies that, for any $r,\eta\in [0,\varepsilon_0]$, with $\varepsilon_0$ sufficiently small, we have
 \begin{equation}
     \gamma(r\eta)\leq \varphi(\eta)\gamma(r) \quad \text{ and }\quad \gamma^{\prime}(r\eta)\leq \frac{1}{r}\Psi(\eta)\gamma(r\eta),
 \end{equation}
where $\varphi$ and $\Psi$ are two Borel functions such that, for some $p\geq 1$ and $\alpha \geq 1$, 
\begin{equation}
    \dint_0^1\log^{p}\left(1+\frac{c_5}{r^{2\alpha}}\right)\varphi(r)\Psi(r)\,dr<\infty.
\end{equation}
Then under these conditions, and by making use of the integration by parts we have

\begin{equation}
    \begin{aligned}
        \dint_0^{\varepsilon_0}\gamma(x\eta)\frac{d\eta}{\eta\sqrt{\log(1/\eta)}}&=-\gamma(\varepsilon_0\, x)\sqrt{\log (1/\varepsilon_0)}+x\, \dint_0^{\varepsilon_0}\sqrt{\log(1/\eta)}\gamma^{\prime}(x\eta)d\eta\\
        &\leq \dint_0^{\varepsilon_0}\sqrt{\log\left(1/\eta\right)}\Psi(\eta)\gamma(x\eta)d\eta\\
        &\leq \gamma(x) \dint_0^{\varepsilon_0}\sqrt{\log\left(1/\eta\right)}\Psi(\eta)\varphi(\eta)d\eta\\
        &\leq c_6\,\gamma(x)  \dint_0^1\log^{p}\left(1+\frac{c_5}{\eta^{2\alpha}}\right)\varphi(\eta)\Psi(\eta)\,d\eta=c_7\, \gamma(x).
    \end{aligned}
\end{equation}
Hence $\gamma$ satisfies \eqref{condition raisonable}.
\end{remark}

\begin{proof}[proof of Lemma \ref{lem upper bound}] 
Since the coordinate processes of $B$ are independent copies of $B_0$, it is sufficient to prove  \eqref{estim small ball 1} when $d=1$. Note that for any $s,t\in I$, we have
\begin{align}
\mathbb{E}\left(B_{0}(s) \mid B_{0}(t)\right)=\frac{\mathbb{E}\left(B_{0}(s) B_{0}(t)\right)}{\mathbb{E}\left(B_{0}(t)^{2}\right)} B_{0}(t):=c(s, t) B_{0}(t).
\end{align}
Then the Gaussian process $(R(s))_{s\in I}$ defined by 
\begin{align}
    R(s):=B_{0}(s)-c(s, t) B_{0}(t),\label{process minus projection}
\end{align}
is independent of $B_{0}(t)$. Let
$$
Z(t, r)=\sup _{s \in B_{\delta}(t, r) \cap I}\left|B_{0}(s)-c(s, t) B_{0}(t)\right|.
$$
Then
\begin{align}
\begin{aligned}
\mathbb{P}&\left\{\inf _{s \in B_{\delta}(t, r) \cap I}\left|B_{0}(s)-z\right| \leq r\right\} \\
&\leq \mathbb{P}\left\{\inf _{s \in B_{\rho}(t, r) \cap I}\left|c(s, t)\left(B_{0}(t)-z\right)\right| \leq r+Z(t, r)+\sup _{s \in B_{\delta}(t, r) \cap I}|(1-c(s, t)) z|\right\}
\end{aligned}\label{upper bound 1}
\end{align}
By the Cauchy-Schwarz inequality and \eqref{commensurate}, we have  for all $s, t \in I$,
\begin{align}\label{estim correlation}
    |1-c(s, t)|=\frac{\left|\mathbb{E}\left[B_{0}(t)\left(B_{0}(t)-B_{0}(s)\right)\right]\right|}{\mathbb{E}\left(B_{0}(t)^{2}\right)} \leq c_1\, \delta(s,t).
\end{align}
Let $r_0:=1/2c_1$, then \eqref{estim correlation} implies that for all $0<r<r_0$ and $s\in B_{\delta}(t,r)\cap I$, we have $1/2\leq c(s,t)\leq 3/2$. Furthermore, for $0<r\leq r_0$, $s\in B_{\delta}(t,r)$, and $z\in [-M,M]^d$, we have $$ |(1-c(s,t))z|\leq c_1\,M\,r.$$
Combining this inequality with \eqref{upper bound 1}, we obtain that for all $z\in [-M,M]^d$

\begin{align}
\begin{aligned}
    \mathbb{P}\left\{\inf _{s \in B_{\delta}(t, r) \cap I}\left|B_{0}(s)-z\right| \leq r\right\}&\leq \mathbb{P}\left\{\left|B_0(t)-z\right|\leq 2\left(M c_1+1\right)r+2 Z(t,r)\right\}\\
    &\leq c_2\left(r+\mathbb{E}\left(Z(t,r)\right)\right),
\end{aligned}
\end{align}
where the last inequality is due to the independence between $B_0(t)$ and $Z(t,r)$, we note also that the constant $c_2$ depends on $M$, $a$, $b$, and $l$ only.

It remains the estimation of the term $\mathbb{E}\left(Z(t,r)\right)$. Let  us consider $d(\cdot,\cdot)$ to be the canonical metric of the centered Gaussian process $R(s)$ defined above in \eqref{process minus projection}. From the Lipschitz property of the projection operator, we obtain that 
\begin{align}
    d(s,s^{\prime})\leq \delta(s,s^{\prime})\label{comaraison metric}
\end{align}
for all $s,s^{\prime}\in B_{\delta}(t,r)\cap I$. Denote by $D:=\sup_{s,s^{\prime}\in B_{\delta}(t,r)\cap I}d(s,s^{\prime})$ the $d$-diameter, and $N_d(B_{\delta}(t,r),\varepsilon)$ the smallest number of $d$-balls of raduis $\varepsilon$ by which we can cover $B_{\delta}(t,r)$. For any fixed $r>0$, we denote by $\varepsilon_0(r)$ the quantity $$ \varepsilon_0(r):=\inf\{\varepsilon>0: N_d\left(B_{\delta}(t,r),\varepsilon\right)<2\},$$ which is immediately smaller than $r$.  From the assumption \eqref{commensurate} and the inequality \eqref{comaraison metric},
we obtain
\begin{align}
    D\leq r\quad \text{ and } \quad N_d(B_{\delta}(t,r),\varepsilon)\leq 2\times\operatorname{1}_{\{\varepsilon_0(r)\leq \varepsilon\leq r\}} + \, \frac{\gamma^{-1}(r\, \sqrt{l})} {\gamma^{-1}(\varepsilon/\sqrt{l})}\times\operatorname{1}_{\{0<\varepsilon<\varepsilon_0(r)\}}.\label{estim entropy}
\end{align}
It follows from \eqref{estim entropy} and the classical entropy upper bound of R. Dudley (see Corollary 4.15 \cite{Adler 2}) that, for some universal constant $c_3$,  
\begin{align}
   \begin{aligned}
    \mathbb{E}\left(Z(t,r)\right)&\leq c_3  \int_{0}^{D}\sqrt{\log N_d\left(B_{\delta}(t,r),\varepsilon\right)}d\,\varepsilon\\
    & \leq c_3\left(\int_0^{\varepsilon_0(r)}\sqrt{\log \frac{\gamma^{-1}(r\, \sqrt{l})}{\gamma^{-1}(\varepsilon/\sqrt{l})}}d\,\varepsilon+\sqrt{\log 2}\int_{\varepsilon_0(r)}^rd\varepsilon\right)\\
    &\leq  c_3\, \left(\sqrt{l}\int_0^{\gamma^{-1}\left( \sqrt{l}\,\varepsilon_0(r)\right)}\sqrt{\log \frac{\gamma^{-1}(r\sqrt{l})}{\eta}}d\gamma(\eta)+\sqrt{\log 2}\, r\right)\\
    &\leq c_3\,\left( \sqrt{\log 2}\, r+ \sqrt{l} \left[ \int_0^{\frac{1}{2}\gamma^{-1}\left( \sqrt{l}\, r\right)}\sqrt{\log \frac{\gamma^{-1}(r\sqrt{l})}{\eta}}d\gamma(\eta) +\sqrt{\log 2}\left( r\sqrt{l}-\gamma\left(\gamma^{-1}(r\sqrt{l})/2\right)\right)\right]\right)
   \end{aligned}\label{estim esp entropy}
\end{align}
where we used only \eqref{estim entropy} and a change of variable. Thanks to the continuity of the process $B_0$, which imply that $\lim_{\eta\rightarrow \infty}\gamma(\eta)\left(\log 1/\eta \right)^{1/2}=0$ (see for example \cite{Talagrand}). Then by using the integration by parts and another change of variables, we get 
\begin{align}
    \int_0^{\frac{1}{2}\gamma^{-1}(r \sqrt{l})}\sqrt{\log \frac{\gamma^{-1}(r\sqrt{l})}{\eta}}d\gamma(\eta)=\gamma\left(\frac{\gamma^{-1}(r\sqrt{l})}{2}\right)\sqrt{\log 2}+\int_0^{1/2}\gamma\left(\gamma^{-1}(r\sqrt{l})y\right)\frac{dy}{y\sqrt{\log(1/y)}}.\label{integration by part}
\end{align}
Then by combining \eqref{integration by part} and \eqref{estim esp entropy}, we get 
$$ \mathbb{E}\left(Z(t,r)\right)\leq c_3\, (l+1) f_{\gamma}(r),$$
which give that
\begin{align}
    \mathbb{P}\left\{\inf _{ s\in B_{\delta}(t,r)\cap I}\|B(s)-z\| \leqslant r\right\} \leqslant c_4 (r+f_{\gamma}(r))^{d},
\end{align}
where 
$c_4$ depends on $a$, $b$, $l$ and $M$. This finishes the proof. 
\end{proof}

\section{Hausdorff dimension of the image set $B(E)$}
\subsection{Hausdorff measure and dimension associated to the canonical metric}
Before giving an explicit formula for the Hausdorff dimension of the image $B(E)$ under some general conditions on $\gamma$ 
, we need first to define the Hausdorff measure and dimension associated to the canonical metric $\delta$ of the process $B$. For $\beta>0$ and $E \subset \mathbb{R}_+$, the $\beta$-dimensional Hausdorff measure of $E$ with respect to the metric $\delta$ is defined by 
\begin{equation}
\mathcal{H}_{\delta}^{\beta}(E)=\lim_{\eta \rightarrow 0}\inf \left\{\sum_{n=1}^{\infty}\left(2 r_{n}\right)^{\beta}: E \subseteq \bigcup_{n=1}^{\infty} B_{\delta}\left(r_{n}\right), r_{n} \leqslant \eta \right\}.\label{parabolic Hausdorff measure 0}
\end{equation}
The Bessel-Riesz type capacity of order $\beta$ on  the metric space $(\mathbb{R}_+,\delta)$ is defined by 
	\begin{equation}
	\mathcal{C}_{\delta,\beta}(E)=\left[\inf _{\nu \in \mathcal{P}(E)} \mathcal{E}_{\delta,\beta}(\nu)\right]^{-1},\label{delta capacity}
	\end{equation}
	where $\mathcal{P}(E)$ is the family of probability measures carried by $E$, and $\mathcal{E}_{\delta,\beta}(\nu)$ denote the $\beta$-energy of a measure $\nu \in \mathcal{P}(E)$ in the metric space $(\mathbb{R}_+,\delta)$, which is defined as  $$ \mathcal{E}_{\delta,\beta}(\nu):= \int_{\mathbb{R}_+} \int_{\mathbb{R}_+} \varphi_{\beta}(\delta(t,s)) \nu(d t) \nu(d s),$$
	and the function $\varphi_{\beta}:(0,\infty)\rightarrow (0,\infty)$ is defined by
	\begin{equation}
	\varphi_{\beta}(r)=\left\{\begin{array}{ll}{r^{-\beta}} & {\text { if } \beta>0} \\ {\log \left(\frac{e}{r \wedge 1}\right)} & {\text { if } \beta=0} \\ {1} & {\text { if } \beta<0}\end{array}\right..\label{radial kernel}
	\end{equation}
The $\delta$-Hausdorff dimension associated to the Hausdorff measure $\mathcal{H}_{\delta}^{\beta}(\cdot)$ is defined as 
  \begin{align}
  \begin{aligned}
  \dim_{\delta}(E):=\sup\left\{\beta: \mathcal{H}_{\delta}^{\beta}(E) >0 \right\}.
  \end{aligned}\label{dim delta}
  \end{align}
There exists also an alternative expression given through the Bessel-Riesz capacities by
\begin{align}\label{altern dim delta}
    \dim_{\delta}(E)=\sup\left\{\beta: \mathcal{C}_{\delta,\beta}(E)>0 \right \},
\end{align}
Notice that the lower inequality (i.e. $\dim_{\delta}(E)\geq \sup\left\{\beta: \mathcal{C}_{\delta,\beta}(E)>0 \right \}$) follows from an application of the energy method (see for example Theorem 4.27 in \cite{Peres&Morters}), and the upper inequality holds from an application of the Frostman's Lemma in the metric space $(\mathbb{R}_+,\delta)$. Indeed, It was proven in \cite{Ho95} that for any general metric space $(E,\delta)$, we have  
\begin{align}\label{dim Frost}
     \dim_{\delta}(E)=\sup\left\{\beta / \exists r_0>0, c_0>0, \text{ and } \nu \in \mathcal{P}(E) \text{ s. t.}\quad \nu\left(B_{\delta}(x,r)\right)\leq c_0\, r^{\beta} \text{ for all } r<r_0,\, x\in E \right\}.
\end{align}
One can see for example Proposition $5$ and  Note $12$ in \cite{Ho95} for a good understanding of this last formulation. Then considering this definition, we can prove now the remaining inequality in \eqref{altern dim delta}. 
Let $\alpha<\dim_{\delta}(E)$, and we fix some $\beta\in (\alpha,\dim_{\delta}(E))$. By using the equality \eqref{dim Frost}, there exists $\nu \in \mathcal{P}(E)$, $0<r_0<1$, and $0<c_0<\infty$ such that $\nu\left(B_{\delta}(x,r)\right)\leq c_0\, r^{\beta}$ for all $r<r_0$. For a fixed $t\in E$, since $\nu$ has no atom, we make the following decomposition
\begin{align}\label{decomp energy}
    \begin{aligned}
    \int_E\frac{\nu(ds)}{\delta(t,s)^{\alpha}}=\sum_{k=1}^{\infty}\int_{\delta(t,s)\in (2^{-k},2^{-k+1}]}\frac{\nu(ds)}{\delta(t,s)^{\alpha}}&\leq \sum_{k=1}^{\infty}2^{k\alpha}\nu\left(B_{\delta}(t,2^{-k+1})\right)\\
    &\leq c_1\, \sum_{k=1 }^{\infty}2^{-k(\beta-\alpha)},
    \end{aligned}
\end{align}
with $c_1=2^{\beta}\,c_0$. Since $\alpha<\beta$ the last sum is finite and independent of $t\in E$. Hence
$$\mathcal{E}_{\delta,\alpha}(\nu):=\int\int\frac{\nu(ds)\nu(dt)}{\delta(t,s)^{\alpha}}<\infty,$$
which finishes the proof.

\begin{remark}\label{examples of delta-dim}
To illustrate some interesting cases that are covered by our study, we consider:
\begin{itemize}
    \item[i)] For $\beta >0$ and $\gamma$ defined near $0$ by  $$ \gamma(r):=r\,\left(\log\left(1/r\right)\right)^{\beta}.$$
    First, we remarque that under \eqref{commensurate}, for all $\eta>0$ small enough and for all $s,t\in [0,1]$ such that $|t-s|\leq \varepsilon_0$ we have $$ \frac{1}{\sqrt{l}}|t-s|\leq \delta(s,t) \leq \sqrt{l}|t-s|^{1-\eta},$$
which implies imediately that, 
$ \dim_{Euc}(E)\leq \dim_{\delta}(E)\leq \dim_{Euc}(E)/1-\eta$ for all $\eta>0$. By making $\eta \downarrow 0$ implies that $$\dim_{\delta}(E)=\dim_{Euc}(E),$$
where $\dim_{Euc}(\cdot)$ denote the Hausdorff dimension associated to the Euclidean metric on $\mathbb{R}_+$.
\item[ii)]\textit{Hölder scale}: For $\beta \in \mathbb{R}$ and $H\in (0,1)$ and $\gamma$ defined near $0$ by $$ \gamma(r)=r^H\left(\log\left(1/r\right)\right)^{\beta}.$$
By the same argument as in $(i)$, we can verify that (in both cases $\beta>0$ and $\beta<0$) $\dim_{\delta}(E)=\frac{\dim_{Euc}(E)}{H}$.
\item[iii)]\textit{Logarithmic scale}: This is an interesting case, which need to be studied carefully, because it is the most irregular case. For $\beta>1/2$ and $\gamma$ defined near $0$ by $$ \gamma(r)=\frac{1}{\left(\log\left(1/r\right)\right)^{\beta}}.$$
First, let us fix $E\subset [0,1]$ to be a Borel set such that $\dim_{\delta}(E)<\infty$, then using the fact that $r^{\alpha}=o\left(\gamma(r)\right)$ for any $\alpha>0$, we get that $$\dim_{Euc}(E)\leq \alpha\,\dim_{\delta}(E) \text{ for all $\alpha>0$ }.$$
Then by letting $\alpha\downarrow 0$ we obtain that $\dim_{\delta}(E)=0$. Hence the Euclidean scale is not sufficient to describe the geometry of some Borel sets. and we will see later that in order to describe a lot of geometric properties of the logarithmic Brownian motion, we need to restrict the process to the class of subsets $E\subset [0,1]$ with $\dim_{Euc}(E)=0$. 
\end{itemize}
\end{remark}
In order to better understand that the size of sets of finite and positive $\delta$-Hausdorff dimension is totally dependent on the metric $\delta$,
we construct, for some fixed $\zeta>0$, a Borel subset $E\subset [0,1]$ such that $\dim_{\delta}(E)=\zeta$ with $0<\mathcal{H}^{\zeta}(E)<\infty$. This will be also  helpful to understand the difference between the different scales defined in the previous remark. 
\begin{lemma}
Let $\zeta>0$, 
Then there exists a compact subset $C_{\zeta}$ of $[0,1]$, such that $0<\mathcal{H}^{\zeta}_{\delta}(C_{\zeta})<\infty$. 
\end{lemma}
\begin{proof}
 Let $\delta^{*}$ be the metric defined as $\delta^{*}(t,s)=\gamma(|t-s|)$. We know from $\eqref{commensurate}$ that the metrics $\delta$ and $\delta^*$ are commensurate, then it will be sufficient to construct a compact set $C_{\zeta}$ such that $0<\mathcal{H}_{\delta^*}^{\zeta}\left(C_{\zeta}\right)<\infty$. We will construct a $\delta^*$-generalised Cantor set of $\delta^*$-Hausdorff dimension equal to $\zeta$. Indeed, let $I_0\subset [0,1]$ an interval of length $\varepsilon_0$. Let first $t_1=\gamma^{-1}\left(2^{-1/\zeta}\right)$ and $l_1=t_1$,   and let $I_{1,1}$ and $I_{1,2}$ two subintervals of $I_0$ with length $l_1\,\varepsilon_0$. For $k\geq 2$, we construct $t_k$, $l_k$, and $\{I_{k,j}:j=1,...,2^{k}\}$ inductively, in the following way: $t_k=\gamma^{-1}\left(2^{-k/\zeta}\right)$ and $l_k=t_k/t_{k-1}$, and the intervals $I_{k,1},...,I_{k,2^k}$ are constructed by conserving only two intervals of length $l_k\, |I_{k-1,i}|=t_k\, \varepsilon_0$ from each interval $I_{k-1,i}$ of the previous iteration. We define $C_{\zeta,k}$ to be the union of the intervals $\left(I_{k,j}\right)_j$ of each iteration. The compact set $C_{\zeta}$ is defined to be the limit set of this construction, namely we have 
 \begin{equation}
     C_{\zeta}:=\bigcap_{k=1}^{\infty}C_{\zeta,k}   .
 \end{equation}
 It remains to show that $0<\mathcal{H}_{\delta^*}^{\zeta}(C_{\zeta})<\infty$. For the upper bound we use the fact that, for all $k$ fixed, the family $\left(I_{k,j}\right)_{j\leq 2^k}$ is a covering of $C_{\zeta}$ and each $I_{k,j}$ is contained in a an open ball $B_{\delta^*}\left(s_i,2^{-k/\zeta}\right)$. Then by definition of the $\zeta$-dimensional Hausdorff measure \eqref{parabolic Hausdorff measure 0} we have
 \begin{equation}
     \mathcal{H}_{\delta^*}^{\zeta}(C_{\zeta})\leq \sum_{j=1}^{2^k}\left(2\,2^{-k/\zeta}\right)^{\zeta}=2^{\zeta}.
 \end{equation}
For the lower bound, we define a measure $\nu$ on $C_{\zeta}$ by the mass distribution principle (\cite{Falconer}), For any $k\geq 1$, we define 
\begin{equation}\label{mass distr}
    \nu(I_{k,i})=2^{-k} \quad \text{ for $i=1,\ldots,2^k$}
\end{equation}
and $\nu\left([0,1] \setminus C_{\zeta,k}\right)=0$. Then by Proposition 1.7 in \cite{Falconer}, $\nu$ can be extended to a probability measure on $C_{\zeta}$. For $t\in C_{\zeta}$ and $r>0$ small enough, let $k\geq 1$ such that $2^{-(k+1)/\zeta}<r\leq 2^{-k/\zeta}$, then it is easy to check that, the ball $B_{\delta^*}(t,r)$ intersect at most $4$ interval $I_{k,i}$, which by using \eqref{mass distr} imply that 
\begin{equation}
    0<\nu\left(B_{\delta^*}(t,r)\right)\leq  2^{-k+2}\leq 8\,r^{\zeta}.
\end{equation}
Then by using the mass distribution principle (see \cite{Falconer} pg. 60), we deduce that $\mathcal{H}_{\delta^*}^{\zeta}\left(C_{\zeta}\right)\geq \nu\left(C_{\zeta}\right)/8=1/8$, which finishes the proof.
\end{proof}
\subsection{Hausdorff dimension for the rang set $B(E)$}

Now our aim is to give, under some weaker assumptions on $\gamma$, an upper and lower bounds for the Hausdorff dimension of the image of a Borel set $E$ by the Gaussian process $B$. We notice that when $B$ has stationary increments, an explicit formula for $\dim_{Euc}\left(B(E)\right)$ was given in terms of $d$ and $\dim_{\delta}(E)$\footnote{Another approach was given in \cite{Hawkes} in order to define the term $\dim_{\delta}(E)$, but this approach seems to be valid only when $B$ has stationary increments, instead of our approach which does not require the stationarity of the increments} by Hawkes in Theorem 2 in \cite{Hawkes}, under the strong condition $ind(\gamma)>0$, where 
$ind\left(\gamma\right)$ denote the lower index of the function $\gamma$, which is defined as
\begin{align}\label{def index}
    \begin{aligned}
    ind\left(\gamma\right):&=\sup\{\alpha: \gamma(x)=o\left(x^{\alpha}\right)\}\\
    &=\left(\inf\{\beta: \gamma(x)=o\left(x^{1/\beta}\right)\} \right)^{-1}.
    \end{aligned}
\end{align}
Firstly, we seek condition on $\gamma$ that can weaken the condition $ind\left(\gamma\right)>0$ and which could help us to build an appropriate covering of $B(E)$. It is worth noting that the condition \eqref{condition raisonable} is important for the construction of some covering for $B(E)$, in order to be able to provide an upper bound for its Hausdorff dimension $\dim_{Euc}\left(B(E)\right)$. But to be on the right path of generalizing Theorem 2 in \cite{Hawkes}, we need a condition which should be satisfied by all functions $\gamma$ with  $ind(\gamma)>0$. Even though \eqref{condition raisonable} is already satisfied by all power functions, and some other important examples of interest (see Proposition \ref{prop RVF}, Examples \ref{exple of RVF}, and Remark \ref{ex satisfy reas cond}), we are not able to show that is satisfied by all continuous functions with strictly positive lower index. Nevertheless, we can provide another condition which is weaker than \eqref{condition raisonable}, and we will show that it is satisfied by all functions $\gamma$ of strictly positive lower index that we might work with. It will be also useful to provide an optimal upper bound for $\dim_{Euc}\left(B(E)\right)$. Indeed, we state following condition: 

For all $0<\varepsilon<1$ sufficiently small, there exist two constants $c_{1,\varepsilon}>0$ and $x_{\varepsilon}>0$, such that
\begin{align}\label{nice condition}
    \dint_0^{1/2}\gamma(xy)\frac{dy}{y\sqrt{\log(1/y)}}\leq c_{1,\varepsilon}\, \left(\gamma(x)\right)^{1-\varepsilon} \quad \text{ for all $0<x<x_{\varepsilon}$.}
\end{align}
Therefore, this last condition is immediately weaker than condition \eqref{condition raisonable}. 
Now we provide proof of the comforting fact that all functions $\gamma$ with a positive finite lower index that we might work with are satisfying \eqref{nice condition}.
\begin{lemma}\label{lem check nice cond}
Let $\gamma$ be continuous, strictly increasing, and concave near the origin. If we assume that $ind(\gamma)\in (0,\infty)$, then $\gamma$ satisfies the condition \eqref{nice condition}.
\end{lemma}
\begin{proof}
Let $\alpha:=ind(\gamma)$, and we fix $\varepsilon>0$ small enough, then there exists a constant $c_{2,\varepsilon}$ such that $\gamma(x)\leq c_{2,\varepsilon}\, x^{\alpha-\varepsilon}$, for any $x\in [0,1/2]$. We have also the existence of another constant $c_{3,\varepsilon}$ and a sequence $(x_n)_n$ decreasing to $0$ such that $\gamma(x_n)\geq c_{3,\varepsilon}\, x_n^{\alpha+\varepsilon}$ for all $n$. We may assmue without loss of generality that $\gamma(x)\geq c_{3,\varepsilon}\, x^{\alpha+\varepsilon}$ for all $x\in (0,1/2]$.
 We now only need to show that for some $0<x_{\varepsilon}<1$ small enough, we have 
\begin{align}\label{check weak cond}
    I:=\frac{1}{\gamma(x)}\dint_0^{1/2}\gamma(xy)\frac{dy}{y\, \sqrt{\log(1/y)}}\leq c_{4,\varepsilon}\, \gamma^{-\varepsilon}\left(x\right),
\end{align}
for all $0<x<x_{\varepsilon}$. For $x<1/2$, we splite the above integral into intervals $(0,x]$ and $(x,1/2]$, and using the fact that $\gamma$ is increasing as well as the bounds obtained above on $\gamma$, we have 
\begin{align*}
    I&=\dint_0^{x}\frac{\gamma(xy)}{\gamma(x)}\frac{dy}{y\sqrt{\log(1/y)}}+\dint_x^{1/2}\frac{\gamma(xy)}{\gamma(x)}\frac{dy}{y\sqrt{\log(1/y)}}\\
    &\leq \frac{c_{2,\varepsilon}}{c_{3,\varepsilon}}x^{-2\varepsilon}\dint_0^{x}y^{\alpha-\varepsilon-1}dy+\frac{\gamma(x/2)}{\gamma(x)}\,\dint_x^{1/2}\frac{dy}{y\sqrt{\log(1/y)}}\\
    &\leq c_{5,\varepsilon}\,\left( x^{\alpha-3\varepsilon}+\sqrt{\log(1/x)}\right).
\end{align*}
By choosing $\varepsilon<\alpha/3$, we get that $I \leq 2c_{5,\varepsilon}\, \sqrt{\log(1/x)}$. Using again the fact that $\gamma$ has a positive lower index, we get that $\sqrt{\log(1/x)}=o\left( \gamma^{-\varepsilon}(x)\right)$, which gives the desired inequality in \eqref{check weak cond}.
\end{proof}
Now we can state the main result of this section
\begin{theorem}\label{Hausd dim image}
Let $B$ the continuous $\mathbb{R}^d$-valued centered Gaussian process defined above such that the canonical metric $\delta$ satisfies \eqref{commensurate} with a function $\gamma$ that satisfies Hypothesis \ref{H1}. For any Borel set $E\subset [0,1]$, we have 
\begin{itemize}
    \item[i)] 
    \begin{align}\label{dim lower bnd}
    \dim_{Euc}(B(E)) \geq \min\left(d,\dim_{\delta}(E)\right) \text{ a.s.}
    \end{align}
    \item[ii)] Under the additional condition \eqref{nice condition} we have 
    \begin{align}\label{dim upper bnd}
    \dim_{Euc}(B(E))= \min\left(d,\dim_{\delta}(E)\right)\text{ a.s.}   
    \end{align}
    where $\dim_{Euc}(\cdot)$ denote the Hausdorff dimension associated to the Euclidean metric.
\end{itemize}
\end{theorem}
Before proving this theorem let us introduce some notations. Let $\mathfrak{C}=\bigcup_{n=0}^{\infty}\mathfrak{C}_n$ be the class of all $\gamma$-dyadic intervals such that every $C\in \mathfrak{C}_n$ has the forme $$ [(j-1)\gamma^{-1}(2^{-n}), j\gamma^{-1}(2^{-n})],$$ for $k,n \in \mathbb{N}$. By using \eqref{commensurate} and substituting  $\delta$-balls by $\gamma$-dyadic intervals in the definition of Hausdorff measure, we obtain another family of outer measures $\left(\widetilde{H}_{\delta}^{\beta}(\cdot)\right)_{\beta}$. Fortunately by making use of \eqref{commensurate} we can check that for all fixed $\beta$, the measures $\mathcal{H}_{\delta}^\beta(\cdot)$ and $\widetilde{H}_{\delta}^{\beta}(\cdot)$ still equivalent.  The proof follows from the same lines as in Taylor and Watson \cite{Tay&Wats} p. 326. A necessary condition is that, $\gamma(2s)\leq c\,\gamma(s)$ for all $0<s<\varepsilon_0$ with some constant $c>0$, which is an immediate consequence of the concavity of $\gamma$ near zero.

\begin{proof}[proof of Theorem \ref{Hausd dim image}] We begin by proving $(i)$. First, by the countable stability of Hausdorff dimension we can suppose without loss of generality that $\operatorname{diam}(E)\leq \varepsilon$. Let $\zeta<d\wedge \dim_{\delta}(E)$, then \eqref{altern dim delta} implies that there is a probability measure $\nu$ supported on $E$ such that 
\begin{align}
  \int_E\int_E\frac{\nu(ds)\nu(dt)}{\delta\left(s,t\right)^{\zeta}}<\infty.
\end{align}
Let $\mu:= \nu \circ B^{-1}$ be the image of $\nu$ by the process $B$, then
\begin{align}
    \begin{aligned}
    \mathbb{E}\left(\int_{\mathbb{R}^d}\int_{\mathbb{R}^d}\frac{\mu(dx)\mu(dy)}{\|x-y\|^{\zeta}}\right)&=\int_E\int_E\mathbb{E}\left(\frac{1}{\|B(t)-B(s)\|^{\zeta}}\right)\nu(ds)\nu(dt)\\
    &=c_{\zeta}\, \int_E\int_E\frac{\nu(ds)\nu(dt)}{\delta(t,s)^{\zeta}}<\infty,
    \end{aligned}
\end{align}
where $c_{\zeta}=\mathbb{E}\left(1/\|X\|^{\zeta}\right)$ with $X \sim \mathcal{N}(0,I_d)$, which is finite because $\zeta<d$. Then Frostman's theorem, on $\mathbb{R}^d$ endowed with the Euclidean metric, implies that $\mathcal{C}^{\zeta}(B(E))>0$ a.s. 
Hence $\dim_{Euc}\left(B(E) \right)\geq \zeta$, and by making $\zeta \uparrow d\wedge \dim_{\delta}(E)$ we get the desired inequality. Let us now prove the upper bound part $(ii)$. Indeed, we suppose that $d> \dim_{\delta}(E)$ otherwise there is nothing to prove. Let $\zeta>\dim_{\delta}(E)$, by definition of Hausdorff dimension we have $\widetilde{\mathcal{H}}_{\delta}^{\zeta}(E)=0$. Let $\eta>0$, so that there is a family of $\gamma$-dyadic interval $(C_k)_{k\geq 1}$ such that for every $k\geq 1$ there is $n_k,j_k\in \mathbb{N}$ and $C_k$ has the form $[(j_k-1)\,\gamma^{-1}\left(2^{-n_k}\right),j_k\,\gamma^{-1}\left(2^{-n_k}\right)]$ and we have
\begin{align}
    E\subset\bigcup_{k=1}^{\infty}C_k \quad \text{ and }\quad \sum_{k=1}^{\infty}\left|C_k\right|_{\delta
    }^{\zeta}< \eta,\label{cover Hausdorff 1}
\end{align}
where $\left|\cdot\right|_{\delta}$ denote the diameter associated to the metric $\delta$. By using \eqref{commensurate} it is easy to verify that $ c_1\, 2^{-n_k} \leq \left|C_k\right|_{\delta}\leq c_2\, 2^{-n_k}$, where $c_1$ and $c_2$ depend on $l$ only. For all  fixed $n\geq 1$, let $M_n$ be the number of indices $k$ for which $n_k=n$, 
  implies that 
\begin{equation*}
    \sum_{n=1}^{\infty}M_n 2^{-n\,\zeta}<\frac{\eta}{c_1}.
\end{equation*}

Let $K\subset \mathbb{R}^d$ an arbitrary compact set, we will construct an adequate covering of $B\left(E\right)\cap K$. To simplify we suppose that $K=[0,1]^d$. For every $n\geq 1$ let $\mathcal{I}_n$ be the collection of Euclidean dyadic subcubes of $[0,1]^d$ of side length $2^{-n}$, and for all $i=1,...,M_n$ let $\mathcal{G}_{n,i}$ be the collection of cubes $I\in \mathcal{I}_n$ such that $B\left(C_i^{n}\right)\cap I\neq \emptyset$. 
Then we have 

\begin{equation}
    \begin{aligned}\label{inclusion intersection 0}
    \begin{array}{c}
        B\left(E\right)\cap [0,1]^d\subseteq \bigcup_{n=1}^{\infty}\,\,\bigcup_{i=1}^{M_n}\bigcup_{\,\,I\in \mathcal{G}_{n,i}}B\left(C_i^{n}\right)\cap I.
    \end{array}
    \end{aligned}
\end{equation}
For all $n\geq 1$, $i\in \{1,...,M_n\}$, and $ I\in \mathcal{I}_n$. Let $\varepsilon>0$ small enough, by using  \eqref{estim small ball 1} and condition \eqref{nice condition}, we get that  
$$ \mathbb{P}\left\{I\in \mathcal{G}_{n,i}\right\}\leq c_3 2^{-n\,(1-\varepsilon)d},$$
where $c_3$ may depends on $\varepsilon$, but not on $n$. We denote by $\mathcal{H}^{\zeta}_{\infty}(\cdot)$ the $\zeta$-Hausdorff content. It is known that the Hausdorff dimension is defined also through Hausdorff contents in the same way as Hausdorff measures, one can see Proposition 4.9 in \cite{Peres&Morters} for the proof of this fact. Then we obtain
\begin{align}
\begin{aligned}
\mathbb{E}\left(\mathcal{H}_{\infty}^{\zeta+\varepsilon\,d}\left(B(E)\cap [0,1]^d \right)\right)&\leq c_4\, \sum_{n=1}^{\infty}\sum_{i=1}^{M_n}\sum_{I\in \mathcal{I}_n}2^{-n\left(\zeta+\varepsilon\,d\right)}\mathbb{P}\{I\in \mathcal{G}_{n,i}\}\\
&\leq c_5 \sum_{n=1}^{\infty}M_n\operatorname{card}(\mathcal{I}_n)\times 2^{-n(d+\zeta)}\\
        &=c_5\sum_{n=1}^{\infty}M_n 2^{-n\, \zeta}<c_6\,\eta,\label{estim content}
\end{aligned}
\end{align}
where the constants $c_4$, $c_5$, and $c_6$ depend on $\varepsilon$ only. Since $\eta>0$ is arbitrary we get that $\mathcal{H}_{\infty}^{\zeta+\varepsilon\,d}\left(B(E)\cap K \right)=0$ a.s. and then $\dim_{Euc}(B(E)\cap [0,1]^d)\leq \zeta+\varepsilon\,d$ a.s. Hence by making $\zeta\downarrow \dim_{\delta}(E)$ and $\varepsilon\downarrow 0$, we obtain that $\dim_{Euc}(B(E)\cap K)\leq \dim_{\delta}(E)+\varepsilon\,d$, and the desired inequality follows by making $\varepsilon\downarrow 0$. So, the using the countable stability property of Hausdorff dimension ensures that $ \dim\left( B(E)\right)\leq \dim_{\delta}(E)$, which finishes the proof of $(ii)$. 
\end{proof}

\begin{remark}
Notice that condition \eqref{nice condition} fails to holds in the logarithmic scale. But we still get a $\sqrt{\log}$ correction, precisely we get $$\int_0^{1/2}\gamma(xy)\frac{dy}{y\sqrt{\log(1/y)}} \leq \gamma(x)\, \sqrt{\log(1/x)}.$$
Hence, by using Lemma \ref{lem upper bound} we obtain that 
\begin{equation}\label{small ball log scale}
		\mathbb{P}\left\{\inf _{ s\in B_{\delta}(t,r)\cap I}\|B(s)-z\| \leqslant r\right\} \leqslant c_2 r^{(1-\frac{1}{2\beta})d},
		\end{equation}
and by following the same lines from \eqref{cover Hausdorff 1} to \eqref{estim content} we get that (the lower bound does not change) $$\dim_{\delta}(E)\wedge d\leq \dim_{Euc}(B(E))\leq \left(\dim_{\delta}(E)+\frac{d}{2\beta}\right)\wedge d. $$
\end{remark}
It is quite remarkable that the irregularity of the process $B$ increase when the lower index $ind(\gamma)$ decrease. When ${ind}(\gamma):=\alpha\in (0,1)$, all trajectories of $B$ are $\beta$-Hölder continuous for all $\beta<\alpha$, and obviously by Lemma \ref{lem check nice cond}, we get that $\gamma$ satisfies \eqref{nice condition}. Hence, the optimal upper bound in \eqref{dim upper bnd} holds immediately. In the other case, when ${ind}(\gamma)=0$, the trajectories of the Gaussian process $B$ are never being H\"older continuous, and no one can be sure if condition \eqref{nice condition} is satisfied in general or not. Since it was shown in the previous remark that the logarithmic scale (i.e. when $\gamma(x):=\log^{-\beta}(1/x)$ for $\beta>1/2$) does not satisfy \eqref{nice condition}, and thinking of this scale as the irregular one, there are several other regularity scales which interpolate between H\"older-continuity and the aforementioned logarithmic scale, this compels us to ask the following question:

\begin{itemize}
\item Is there a continuous and strictly increasing function $\gamma$ with zero index (${ind}(\gamma)=0$) and satisfying the condition \eqref{nice condition}?
\end{itemize}

A positive answer for the above question will is given by the following example, where we will give a class of functions $\left(\gamma_{\alpha}\right)_{\alpha\in (0,1)}$ with zero indexes, such that the weaker condition \eqref{nice condition} is satisfied. Therefore, by Theorem their associated Gaussian processes would satisfy the identity $\dim_{Euc}\left(B(E)\right)=\dim_{\delta}(E)\wedge d$\, a.s.   
\begin{example}\label{exmpl weak cond}
To give an example of $\gamma$ which satisfies the condition \eqref{nice condition}, we consider the family of functions $(\gamma_{\alpha})_{\alpha\in(0,1)}$ defined by $\gamma_{\alpha}(x):=\exp\left(-\log^{\alpha}(1/x)\right)$. It is easy to see that, for any fixed $\alpha\in (0,1)$, $\gamma_{\alpha}(x)$ is less irregular than the logarithmic scale (i.e. $\gamma_{\alpha}(x)=o\left(\log^{-\beta}(1/x)\right)$ for all $\beta>0$), but still more irregular than the Hölder scale (i.e. $x^H=o\left(\gamma_{\alpha}(x)\right)$ for all $H> 0$). It remains to show that $\gamma_{\alpha}$ satisfies $\eqref{nice condition}$. Indeed, we have 
\begin{equation}\label{estim integral 2}
    \begin{aligned}
        \int_{0}^{1 / 2} \gamma_{\alpha}(x y) \frac{d y}{y \sqrt{\log (1 / y)}}&=\int_{0}^{1 / 2} \exp\left(-\left(\log(1/x)+\log(1/y)\right)^{\alpha}\right) \frac{d y}{y \sqrt{\log (1 / y)}}\\
        &=\int_{\log 2}^{\infty}\exp\left(-\left(\log(1/x)+z\right)^{\alpha}\right)\frac{dz}{\sqrt{z}},
    \end{aligned}
\end{equation}
where we used the change of variable $z=\log(1/y)$. Using the fact that, for all $\mathfrak{c}\in (0,1)$ there is some $N:=N(\mathfrak{c})>0$ large enough, so that  
\begin{align}
    (1+u)^{\alpha}\geq 1+\mathfrak{c}\,u^{\alpha} \quad \text{for all $u\geq N$, }\label{real estimat}
\end{align}
we may fix $\mathfrak{c}\in (0,1)$, and its corresponding $N(\mathfrak{c})$. Then we break the integral in \eqref{estim integral 2} into the intervals $\,[\log(2), N \log(1/x))\,$ and $\,[N\log(1/x), +\infty)\,$ and denote them by $\mathcal{I}_1$ and $\mathcal{I}_2$, respectively. We write\\ $\left(\log(1/x)+z\right)^{\alpha}=\log^{\alpha}(1/x)\times\left(1+z/\log(1/x)\right)^{\alpha}$, and we note that the second term is bounded from below by $1+\mathfrak{c}\left(\frac{z}{\log(1/x)}\right)^{\alpha}$ when $z\geq N\log(1/x)$ (Thanks to \eqref{real estimat}), and by $1$ when $z< N\log(1/x)$. We first have
\begin{align}\label{I1}
    \mathcal{I}_1\leq\exp\left(-\log^{\alpha}(1/x)\right)\, \int_{0}^{N\log(1/x)}\frac{dz}{\sqrt{z}}=2\,\gamma_{\alpha}(x)\, \sqrt{N\,\log(1/x)}.
\end{align}
On the other hand, we have 
\begin{align}\label{I2}
    \mathcal{I}_2\leq \exp\left(-\log^{\alpha}(1/x)\right)\,\int_{0}^{\infty}\operatorname{e}^{-\mathfrak{c}z^{\alpha}}\frac{dz}{\sqrt{z}}=c_{1,\alpha}\, \gamma_{\alpha}(x).
\end{align}
By combining \eqref{I1} and \eqref{I2}, and the fact that $\,\sqrt{\log(1/x)}=o\left(\gamma_{\alpha}^{-\varepsilon}(x)\right)$ for all $\varepsilon>0$, we obtain the estimation of \eqref{nice condition}.
\end{example}

\section{Hitting probabilities}
\subsection{Preliminaries}
Our aims now is to develop a criterion for hitting probabilities of a Gaussian process $B$ with canonical metric $\delta$ which satisfies the commensurability condition \eqref{commensurate}. We will establish lower and upper bounds for hitting probabilities in terms of a capacity term and the Hausdorff measure term, respectively. Both of the capacity and Hausdorff measures terms would be  constructed on $\mathbb{R}_+\times \mathbb{R}^d$, and they would be associated to an appropriate metric $\rho_{\delta}$ on $\mathbb{R}_+\times \mathbb{R}^d$, which will be defined bellow. First of all, we define the metric $\rho_{\delta}$ on $\mathbb{R}_+\times \mathbb{R}^d$ by
\begin{align}\label{parabolic metric}
    \rho_{\delta}\left((s,x),(t,y)\right)=\max\{\delta(t,s),\|x-y\|\}, \quad \text{ for all } (s,x),(t,y)\in \mathbb{R}_+\times \mathbb{R}^d.
\end{align}
For an arbitrary $\beta>0$ and $G\subseteq \mathbb{R}_+\times\mathbb{R}^d$, the $\beta$-dimensional Hausdorff measure of $G$ in the metric $\rho_{\delta}$ is defined by 
\begin{equation}
\mathcal{H}_{\rho_{\delta}}^{\beta}(E)=\lim_{\eta \rightarrow 0}\inf \left\{\sum_{n=1}^{\infty}\left(2 r_{n}\right)^{\beta}: E \subseteq \bigcup_{n=1}^{\infty} B_{\rho_{\delta}}\left(r_{n}\right), r_{n} \leqslant \eta \right\}.\label{parabolic Hausdorff measure 1}
\end{equation}
The corresponding Hausdorff dimension of $G$ is defined by
\begin{equation}
    \dim_{\rho_{\delta}}(G)=\inf\{\beta: \mathcal{H}_{\rho_{\delta}}^{\beta}(G)=0\}.
\end{equation}
The $\alpha$-Bessel-Riesz type capacity of $G$ on the metric space $\left(\mathbb{R}_+\times \mathbb{R}^d,\rho_{\delta}\right)$ is defined by 
\begin{equation}
    \mathcal{C}_{\rho_{\delta},\alpha}(G)=
    \left[\inf _{\mu \in \mathcal{P}(E)} \int_{\mathbb{R}_+\times\mathbb{R}^d} \int_{\mathbb{R}_+\times \mathbb{R}^d} \varphi_{\alpha}(\rho_{\delta}(u,v)) \mu(d u) \mu(d v)\right]^{-1},\label{rho-delta capacity}
	\end{equation}
	where the kernel $\varphi_{\alpha}:(0,\infty)\rightarrow (0,\infty)$ is defined in \eqref{radial kernel}. 
	
	It should be kept in mind that the strong condition \eqref{condition raisonable} will be more beneficial than \eqref{nice condition} to study the problem of Hitting probabilities. Specifically, this will be useful to derive an optimal upper bound for $\mathbb{P}\{B(E)\cap F \neq \emptyset\}$ in terms of $\mathcal{H}_{\rho_{\delta}}^d\left(E\times F\right)$. We also note that even though we can not provide a general result which may ensure that any function $\gamma$ with $ind(\gamma)>0$ should satisfy \eqref{condition raisonable}, but at least it is known that all regularly varying functions with index $\alpha\in (0,1)$ are satisfying \eqref{condition raisonable}, see Proposition \ref{prop RVF}. In the other hand, similarly to the case of the weaker condition \eqref{nice condition}, it is worth asking about the existence of an increasing function $\gamma$ with $ind(\gamma)=0$ that satisfies \eqref{condition raisonable}?

    Taking into account the fact that all strictly increasing and continuous functions $\gamma$ with $ind(\gamma)=0$ are highly irregular than any power function near zero. We may conjecture that there is a high possibility of the non-existence of a function $\gamma$ with $ind(\gamma)=0$ which satisfies \eqref{condition raisonable}. Although we have not been able to prove this conjecture ultimately, we would like to observe its high probability to hold. So, we provide a mild sufficient condition for the non-existence of \eqref{condition raisonable}, which is already satisfied -under the Hypothesis \ref{H1}- by a large class of functions with zero lower index as we will see. In order to provide a useful formula for the index of $\gamma$, we will assume that $\gamma$ is differentiable except perhaps at 0, and strictly increasing near 0. Then we have.
    \begin{proposition}\label{prop suffic cond}
    Let $\gamma$ be a differentiable, strictly increasing near 0. We denote by $\Psi_{\gamma}(r):=\frac{r\, \gamma^{\prime}(r)}{\gamma(r)}$. If we assume  $\lim_{r\downarrow 0}\Psi_{\gamma}(r)\log^{1/2}\left(1/r\right)=0$, then 
    \begin{align}\label{cond not hold}
    \lim_{x\downarrow 0}\left(\frac{1}{\gamma(x)}\dint_0^{1/2}\gamma(x\, y)\frac{dy}{y\sqrt{\log(1/y)}}\right)=\infty.
    \end{align}
    \end{proposition}
    Before proving this last proposition, we give the following characterisation of $ind(\gamma)$, when $\gamma$ is a differentiable function.
    \begin{lemma}\label{charact index}
    Let $\gamma$ be a differentiable, strictly increasing, and $\gamma^{\prime}(0^{+})=\infty$. Then we have
    \begin{align}
        \liminf_{r\downarrow 0}\Psi_{\gamma}(r)\leq ind\left(\gamma\right)\leq \limsup_{r\downarrow 0}\Psi_{\gamma}(r).
    \end{align}
    \end{lemma}
    \begin{proof}
    We start by the lower inequality, we suppose that $\liminf_{r\downarrow 0}\Psi_{\gamma}(r)>0$ otherwise there is nothing to prove. Let us fix $0<\alpha^{\prime}<\alpha<\liminf_{r\downarrow 0}\Psi_{\gamma}(r)$, then there is $r_0>0$ such that $\alpha/r\leq \gamma^{\prime}(r)/\gamma(r)$ for any $r\in (0,r_0]$. Next, for $r_1<r_2\in (0,r_0]$ we integrate over $[r_1,r_2]$ both of elements of the last inequality, we obtain that $\log\left(r_2/r_1\right)^{\alpha}\leq \log\left(\gamma(r_2)/\gamma(r_1)\right)$, this implies immediately that $r\mapsto \gamma(r)/r^{\alpha}$ is increasing on $(0,r_0]$, and then $\lim_{r\downarrow 0}\gamma(r)/r^{\alpha}$ exists and finite. Since $\alpha^{\prime}<\alpha$, we get $\lim_{r\downarrow 0}\gamma(r)/r^{\alpha^{\prime}}=0$ and then $\alpha^{\prime}\leq ind(\gamma)$. Considering the fact that $\alpha^{\prime}$ and $\alpha$ are arbitrary, the desired inequality holds by letting $\alpha^{\prime}\uparrow \alpha$ and $\alpha\uparrow \liminf_{r\downarrow 0}\Psi_{\gamma}(r)$. For the upper inequality, we assume that $\limsup_{r\downarrow 0}\Psi_{\gamma}(r)<\infty$, and we fix $\alpha^{\prime}>\alpha>\limsup_{r\downarrow 0}\Psi_{\gamma}(r)$ in order to show by a similar argument as above that $r\mapsto \gamma(r)/r^{\alpha}$ is decreasing near 0, and $\lim_{r\downarrow 0}\gamma(r)/r^{\alpha}>0$. Hence, $\lim_{r\downarrow 0}\gamma(r)/r^{\alpha}=\infty$ and then by letting $\alpha^{\prime}\downarrow \alpha$ and $\alpha\downarrow \limsup_{r\downarrow 0}\Psi_{\gamma}(r)$ the desired inequality is obtained.
    \end{proof}
    \begin{proof}[Proof of Proposition \ref{prop suffic cond}]First, we note that the proposition's assumption implies that $\lim_{r\rightarrow 0}\Psi_{\gamma}(r)=0$, and thanks to Lemma \ref{charact index}, which ensures $ind\left(\gamma\right)=0$. Now, using the change of variable $z=xy$, it's easy to check that  
    \begin{align}\label{lwr bnd integr 1}
     \dint_0^{1/2}\gamma(x\, y)\frac{dy}{y\sqrt{\log(1/y)}}\geq \dint_0^{x/2}\gamma( z)\frac{dz}{z\sqrt{\log(1/z)}}.
 \end{align}
We denote by $\Phi(x):=\int_0^{x}\gamma( z)\frac{dz}{z\sqrt{\log(1/z)}}$. Dividing both of the terms in \eqref{lwr bnd integr 1} by $\gamma(x)$, and we use the fact that, for some constant $c_1>0$ we have $\gamma(x)\leq c_1\, \gamma(x/2)$ for all $x>0$ sufficiently small, we obtain 
\begin{align}
    \begin{aligned}
    \frac{1}{\gamma(x)}\dint_0^{1/2}\gamma(x\, y)\frac{dy}{y\sqrt{\log(1/y)}}
    \geq c_1^{-1}\frac{\Phi(x/2)}{\gamma(x/2)}.
    \end{aligned}
\end{align}
Then it will be sufficient to show that $\lim_{x\downarrow 0}\Phi(x)/\gamma(x)=\infty$. Indeeed, since $\Phi(0^+)=\gamma(0^+)=0$, and the functions $\Phi$ and $\gamma$ are differentiable near 0, such that $\gamma^{\prime}(x)>0$ pour tout $0<x<x_0$, for some small $x_0>0$, and by using the assumption, we derive that $\lim_{x\downarrow 0}\Phi^{\prime}(x)/\gamma^{\prime}(x)=\lim_{x\downarrow 0}1/\Psi_{\gamma}(x)\,\log^{1/2}(1/x)=\infty
$. Therefore, an application of the hospital rule yields that
$$ 
\lim_{x\downarrow 0}\frac{\Phi(x)}{\gamma(x)}=\lim_{x\downarrow 0}\frac{\Phi^{\prime}(x)}{\gamma^{\prime}(x)}=+\infty,
$$
which finishes the proof of \eqref{cond not hold}.
    \end{proof}
To realize the usefulness of this sufficient condition, it is enough to check it on some examples
\begin{example}
\textbf{(i)} $\gamma(x)=\log^{-\beta}(1/x)\times m(x)$, with $\beta\geq 1/2$, and the function $\operatorname{m}(\cdot)$ admits slower variations than all of  $\,\,\log^{-\alpha}(1/x)$ for any $\alpha>0$, i.e. $\operatorname{m}(r)=o\left(\log^{\alpha}(1/r)\right)$, and such that $\Psi_{m}\left(r\right)=o\left(\log^{-1/2}(1/r)\right)$. After simple calculation we get  $\Psi_{\gamma}(x)=\beta\,\log^{-1}(1/x)+\Psi_m(x).$ Hence,  $\lim_{x\downarrow 0}\Psi_{\gamma}(x)\log^{1/2}(1/x)=0$.  We can consider $m(x):=\log^{\alpha}\left(\log(1/x)\right)$ with $\alpha\in \mathbb{R}$ when $\beta>1/2$. But when $\beta=1/2$, we should choose $\alpha<0$ in order to coserve the continuity of the Gaussian process $B$. In particular when $\alpha=0$, it can be deduced that $\gamma(x)=\log^{-\beta}(1/x)$ satisfies \eqref{cond not hold} for any $\beta>1/2$.\\
\textbf{(ii)} $\gamma(x)=\exp\left(-\log^{\alpha}(1/x)\right)$ with $0<\alpha<1$. Therefore  $$\lim_{x\downarrow 0}\Psi_{\gamma}(x)\log^{1/2}(1/x)=0 \quad \text{ if and only if }\quad 0<\alpha<1/2.$$ So in that last case $\gamma$ satisfies \eqref{cond not hold}. 
    Otherwise, the case $1/2\leq \alpha<1$ still without information! 

\end{example}
\begin{remark}
Notice that the second example above combined with Example \ref{exmpl weak cond} ensure that the function $\gamma_{\alpha}(x)=\exp\left(-\log^{\alpha}(1/x)\right)$ satisfies the weak condition \eqref{nice condition} but does not satisfy the strong one \eqref{condition raisonable}, at least when $\alpha\in (0,1/2)$.    
\end{remark}
\subsection{Criteria for hitting probabilities}
Now, we are ready to present the main results of this section.
\begin{theorem}\label{Hitting proba}
Assume that Hypothesis \ref{H1} holds. Then for all $0<a<b<\infty$ and $M>0$, and for $E\subset [a,b]$ and $F\subset [-M,M]^{d}$ are two Borel sets, we have   
\begin{itemize}
    \item[i)] If the diameter of $E$ is small enough, there exists a constant $C_1>0$ depending only on $a,b, M$ and the law of $B$, such that
    \begin{equation}
    C_1\, \mathcal{C}_{\rho_{\delta},d}(E\times F)\leq \mathbb{P}\left\{B(E)\cap F\neq \emptyset\right\}\label{lower bnd hitting}.
    \end{equation}
    Otherwise, if $\mathcal{C}_{\rho_{\delta},d}(E\times F)>0\,$ then $\mathbb{P}\left\{B(E)\cap F\neq \emptyset\right\}>0 $. 
    \item[ii)] If in addition to the Hypothesis \ref{H1}, the function $\gamma$ satisfies the condition \eqref{condition raisonable}, there exists a constant $C_2>0$ also depending only on $a,b,M$, and the law of $B$, such that 
    \begin{equation}\label{upper bnd hitting}
        \mathbb{P}\left\{B(E)\cap F\neq \emptyset\right\}  \leq C_2\, \mathcal{H}_{\rho_{\delta}}^{d}\left(E\times F\right).
    \end{equation}
\end{itemize}
\end{theorem}
\begin{remark}
We know that the integrability condition $\int_0^{1}\gamma^{-d}(r)dr<\infty$, which was stated in Remark 2.7. in \cite{Eulalia&Viens2013} implied that the process $B$ -restricted to the hull interval $[a,b]$- hits points with positive probability, i.e. $\mathbb{P}\{B([a,b])\ni x\}>0$ for all $x\in \mathbb{R}^d$. But this integrability condition would be largely sufficient in some irregular cases; for example in the logarithmic scale, 
the problem of hitting points might be studied for the process $B$ restricted to some fractal set $E\subset [a,b]$, which could be tiny in the sense that $\dim_{Euc}(E)=0$ but it still has the the capacity of ensuring the non-polarity of points, i.e. $\mathbb{P}\{B(E)\ni x\}>0$. Indeed, the lower bound in \eqref{lower bnd hitting} gives a sharp sufficient condition on $E$ for $B_{|E}$ to hit points. Namely, if $\mathcal{C}_{\delta,d}(E)>0$, then for every $x\in \mathbb{R}^d$ we have $\,\mathbb{P}\{B(E)\ni x\}>0
 $. 
 Then, based on the alternative expression of the Hausdorff dimension $\dim_{\delta}(\cdot)$ in \eqref{altern dim delta}, we deduce that the condition $\dim_{\delta}(E)>d$ is largely sufficient to ensure the non-polarity of points. This gives rise to a generalized integrability condition. Precisely, by using \eqref{dim Frost} and \eqref{decomp energy}, it is easy to check that, for any $0<\varepsilon<\dim_{\delta}(E)-d$ there exists a probability measure $\nu$ supported on $E$ such that $$\int_0^1\gamma^{-d-\varepsilon}(r)\nu(dr)<\infty. $$  

\end{remark}
\begin{remark}
It is worth noting that under the weaker upper condition \eqref{nice condition} we lose the upper bound estimation of the probability $\mathbb{P}\left\{B(E)\cap F\neq \emptyset \right\}$ in terms of $\mathcal{H}_{\rho_{\delta}}^{d}\left(E\times F\right)$. But we still have a weaker bound. Indeed, for all $\varepsilon>0$ small enough, there exists a positive and finite constant $c_{\varepsilon}$ such that we have 
\begin{align}
    \mathbb{P}\left\{B(E)\cap F\neq \emptyset \right\}\leq c_{\varepsilon}\, \mathcal{H}_{\rho_{\delta}}^{d-\varepsilon}(E\times F).
\end{align}
Since for any compact sets $E$ and $F$, the condition $\mathcal{C}_{\rho_{\delta},d}(E\times F)>0$ ensures that $ \mathbb{P}\left\{B(E)\cap F\neq \emptyset\right\}>0$ (without assuming \eqref{nice condition}), we note that the case $\dim_{\rho_{\delta}}(E\times F)=d$ still a critical case even under the weaker condition \eqref{nice condition}, and this is the only critical case that exists. i.e.
$$
\mathbb{P}\left\{B(E)\cap F\neq \varnothing \right\}\left\{\begin{array}{ll}
>0 & \text { if } \dim_{\rho_{\delta}}(E\times F)>d \\
=0 & \text { if } \dim_{\rho_{\delta}}(E\times F)<d
\end{array}\right.,
$$
\end{remark}

\begin{remark}
Notice that when the condition \eqref{nice condition} is not satisfied, there are many cases where the lack of information on the positivity of $\mathbb{P}\left\{ B(E)\cap F\neq \varnothing\right\}$ holds, not only for one point $\dim_{\rho_{\delta}}(E\times F)=d$, but for many values of $dim_{\rho_{\delta}}(E\times F)$. For example, in the logarithmic scale, 
the upper bound holds with the order $d(1-1/2\beta)$ for the $\rho_{\delta}$-Hausdorff measure term, instead of the $\rho_{\delta}$-capacity in the lower bound, which stills holds with the order $d$. Precisely, by using \eqref{small ball log scale}, and by the same covering arguments that will be used in \eqref{cover product} and \eqref{event covering}, we get that

\begin{equation}
    C_1 \mathcal{C}_{\rho_{\delta},d}\left(E\times F\right)\leq \mathbb{P}\left\{B(E)\cap F\neq \varnothing \right\}\leq C_2\, \mathcal{H}_{\rho_{\delta}}^{d(1-1/2\beta)}\left(E\times F\right),
\end{equation}
which tell us that 
$$
\mathbb{P}\left\{B(E)\cap F\neq \varnothing \right\}\left\{\begin{array}{ll}
>0 & \text { if } \dim_{\rho_{\delta}}(E\times F)>d \\
=0 & \text { if } \dim_{\rho_{\delta}}(E\times F)<d(1-1/2\beta)
\end{array}\right.,
$$
 the critical case $ d(1-1/2\beta)\leq \dim_{\rho_{\delta}}(E\times F)\leq d$ still without information.\\
\end{remark}
The following lemma will be used to prove the lower bound part of the hitting probabilities \eqref{lower bnd hitting}. Its proof follows from the same lines as in Lemma 3.2. in \cite{Bierme&Xiao} by using Lemma \ref{lem two pts LND}.   
\begin{lemma}\label{lem lower bnd}
Assume that Hypothesis \ref{H1} holds. Then for all $x,y\in \mathbb{R}^d$ $s,t\in [a,b]$ such that $|t-s|\leq \varepsilon$, we have 

\begin{equation}
\begin{aligned}
\dint_{\mathbb{R}^{2 d}}& e^{-i(\langle\xi, x\rangle+\langle\eta, y\rangle)} \exp \left(-\frac{1}{2}(\xi, \eta)\left(n^{-1
} I_{2 d}+\operatorname{Cov}(B(s), B(t))\right)(\xi, \eta)^{T}\right) d \xi d \eta \\
&\leq c\, \varphi_d\left(\rho_{\delta}\left((s,x),(t,y)\right)\right),
\end{aligned}
\end{equation}
where $I_{2d}$ denote the $2d\times 2d$ identity matrix, $Cov\left(B(s),B(t)\right)$ denote the $2d\times 2d$ covariance matrix of $(B(s),B(t))$, and $\varphi_{d}(\cdot)$ is the kernel defined in \eqref{radial kernel}. 
\end{lemma}

\begin{proof}[proof of Theroem \ref{Hitting proba}] We begin by proving the lower bound in \eqref{lower bnd hitting}. First, let us suppose that the diameter of $E$ is less than $\varepsilon$. Assume that $\mathcal{C}_{\rho_{\delta},d}(E\times F)>0$ otherwise there is nothing to prove. Which implies the existence of a probability measure $\mu\in \mathcal{P}(E\times F)$ such that 
\begin{equation}\label{bound energy}
\mathcal{E}_{\rho_{\delta},d}(\mu):=\int_{\mathbb{R}_+\times\mathbb{R}^d} \int_{\mathbb{R}_+\times \mathbb{R}^d} \varphi_{d}(\rho_{\delta}(u,v)) \mu(d u) \mu(d v)\leq \frac{2}{\mathcal{C}_{\rho_{\delta},d}(E\times F)}.
\end{equation}
Consider the sequence of random measures $(m_n)_{n\geq 1}$ on $E\times F$ defined as 
\begin{align*}
    \begin{aligned}
        m_n(dt dx)&=(2 \pi n)^{d / 2} \exp \left(-\frac{n\|B(t)-x\|^{2}}{2}\right) \mu(d t d x)\\
        &=\int_{\mathbb{R}^{d}} \exp \left(-\frac{\|\xi\|^{2}}{2 n}+i\langle\xi, B(t)-x\rangle\right) d \xi\,  \mu(d x d t).
    \end{aligned}
\end{align*}
Denote the total mass of $m_n$ by $\|m_n\|=m_n(E\times F)$. We want to verify the following claim
\begin{equation}\label{claim 1}
    \mathbb{E}\left(\left\|m_{n}\right\|\right) \geq c_{1}, \quad \text { and } \quad \mathbb{E}\left(\left\|m_{n}\right\|^{2}\right) \leq c_{2} \mathcal{E}_{\rho_{\delta},d}(\mu),
\end{equation}
where the constants $c_1$ and $c_2$ are independent of $n$ and $\mu$.

First, we have 

\begin{equation}
\begin{aligned}
        \mathbb{E}\left(\|m_n\|\right) &=\int_{E\times F} \int_{\mathbb{R}^{d}} \exp \left(-\frac{|\xi|^{2}}{2}\left(\frac{1}{n}+\gamma^{2}(t)\right)-i\langle\xi, x\rangle\right) d \xi \mu(d t d x) \\
& \geq \int_{E\times F} \frac{(2 \pi)^{d / 2}}{\left(1+\gamma^{2}(t)\right)^{d / 2}} \exp \left(-\frac{|x|^{2}}{2 \gamma^{2}(t)}\right) \mu( d t d x) \\
& \geq \frac{(2 \pi)^{d / 2}}{\left(1+\gamma^{2}(b)^{d / 2}\right.} \exp \left(-\frac{d M^{2}}{2 \gamma^{2}(a)}\right)\int_{E\times F} \mu(dt dx)=: c_{1},
\end{aligned}
\end{equation}
This proves the first inequality in \eqref{claim 1}. We have also 
\begin{align}
    \begin{aligned}
\mathbb{E}\left(\left\|m_{n}\right\|^{2}\right)=\int_{(E\times F)^2}  \int_{\mathbb{R}^{2 d}} &e^{-i(\langle\xi, x\rangle+\langle\eta, y\rangle)} \, \\
&\times \exp \left(-\frac{1}{2}(\xi, \eta)\left(n^{-1} I_{2 d}+\operatorname{Cov}(B(s), B(t))\right)(\xi, \eta)^{T}\right) d \xi \, d \eta\,  \mu(dt d x) \mu(d s d y). 
\end{aligned}
\end{align}
We use Lemma \ref{lem lower bnd} and the fact that the diameter of $E$ is less than $\varepsilon$, we get that $\mathbb{E}\left(\|m_n\|^2\right)\leq  \mathcal{E}_{\rho_{\delta},d}(\mu)$, which proves the second inequality in \eqref{claim 1}. 

Now, using the moment estimates in \eqref{claim 1} and the Paley–Zygmund inequality (c.f. Kahane \cite{Kahane}, p.8), one can check that $\{m_n, n\geq 1\}$ has a subsequence that converges weakly to a finite random measure $m_{\infty}$ supported on the set $\{(s,x)\in E \times F : B(s)=x\}$, which is positive on an event of positive probability and also satisfying the moment estimates of \eqref{claim 1}. Therefore, using again the Paley-Zygmund inequality, we conclude that

$$
\mathbb{P}\left\{B(E) \cap F \neq \varnothing\right\} \geq \mathbb{P}\left\{\|m_{\infty}\|>0\right\} \geq \frac{\mathbb{E}(\|m_{\infty}\|)^{2}}{\mathbb{E}\left(\|m_{\infty}\|^{2}\right)} \geq \frac{c_{1}^{2}}{c_{2} \mathcal{E}_{\rho_{\delta},d}(\mu)}.
$$
Hence, \eqref{bound energy} finishes the proof of \eqref{lower bnd hitting}. For the general case. Let us cover $E$ by a countable family of compact sets $(E_i)_{i\geq 1}$ of diameter less than $\varepsilon$. We assume again that $\mathcal{C}_{\rho_{\delta},d}(E\times F)>0$, and let $\mu$ be a probability measure supported on $E\times F$ such that $\mathcal{E}_{\rho_{\delta},d}(\mu)$. This implies that, for all $i\geq 1$ \begin{align}\label{energy cover}
    \int_{E_i\times F} \int_{E_i\times F} \varphi_{d}(\rho_{\delta}(u,v)) \mu(d u) \mu(d v)<\infty.
\end{align}
Since the family $(E_i\times F)_{i\geq 1}$ cover $E\times F$, there exists $i\geq 1$ such that $\mu(E_{i}\times F)>0$. Then the measure $\mu_i(dtdx)=\frac{\mu(dtdx)}{\mu(E_i\times F)}$ is a probability measure  supported on $E_i\times F$ and \eqref{energy cover} implies that $\mathcal{E}_{\rho_{\delta},d}(\mu_i)<\infty$, which ensures that $\mathcal{C}_{\rho_{\delta},d}(E_{i}\times F)>0$, and by using \eqref{lower bnd hitting} we get 
\begin{equation}
   \mathbb{P}\left\{B(E) \cap F \neq \varnothing\right\} \geq \mathbb{P}\left\{B(E_{i}) \cap F \neq \varnothing\right\}\geq C_{i}\,  \mathcal{C}_{\rho_{\delta},d}(E_{i}\times F). 
\end{equation}
which finishes the proof of $(i)$. 

For the upper bound in \eqref{upper bnd hitting}, we use a simple covering argument. We choose an arbitrary constant $\zeta>\mathcal{H}_{\rho_{\delta}}^d(E\times F)$. Then there is a covering of $E\times F$ by balls $\{B_{\rho_{\delta}}((t_i,x_i),r_i), i\geq 1\}$ in $\left(\mathbb{R}_+\times\mathbb{R}^d,\rho_{\delta}\right)$ with small radii $r_i$, such that 
\begin{equation}\label{cover product}
    E\times F\subseteq \bigcup_{i=1}^{\infty}B_{\rho_{\delta}}((t_i,x_i),r_i)\quad  \text{with }\quad  \sum_{i=1}^{\infty}(2r_i)^{d}\leq \zeta.
\end{equation}
It follows that
		\begin{align}\label{event covering}
		\left\{ B(E)\cap F\neq\emptyset\right\}  & \subseteq \bigcup_{i=1}^{\infty}\left\{\,\exists\,(t,x)\in B_{\delta}(t_i,r_i)\times B(x_i,r_i)\,\text{ s.t. } B(t)=x\right\}  \nonumber \\
		& \subseteq\bigcup_{i=1}^{\infty}\left\{ \inf _{ t\in B_{\delta}(t_i,r_i)}\|B(t)-x_i\| \leqslant r_i\right\}. 
		\end{align}
		Since the condition \eqref{condition raisonable} is satisfied, Corollary \ref{cor estim small ball} combined with \eqref{event covering} imply that $ \mathbb{P}\left\{ B(E)\cap F\neq\emptyset\right\}\leqslant c_1\, \zeta$. Let $\zeta\downarrow \mathcal{H}_{\rho_{\delta}}^d(E\times F)$, the upper bound in \eqref{upper bnd hitting} follows. 
\end{proof}

Now we want to provide some optimal lower and upper bounds for the quantity $\mathbb{P}\left\{B(E)\cap F\neq \emptyset \right\}$ in terms of a capacity term of $F$ and a Hausdorff measure term of $F$ with an appropriate order, 
respectively. Let us assume that the Borel set $E$ take some particular form; for example, as a first restriction; we assume that $E$ is $\zeta$-regular set with respect to the metric $\delta$, for some fixed $\zeta>0$, in the sens that there exists a Borel probability measure $\nu$ supported on $E$ such that for some constants $c_1$ and $c_2$, we have  
 \begin{align}\label{zeta-regular set}
     c_1\, r^{\zeta}\leq \nu\left(B_{\delta}(t,r)\right)\leq c_2\, r^{\zeta} \quad \text{ for all } t\in E.
 \end{align}
 Then we have the following result
 \begin{proposition}\label{bnds in trm of F}
 Assume again that Hypothesis \ref{H1} holds. Let $0<a<b<\infty$ and $M>0$, and let $E\subset [a,b]$ be a $\zeta$-regular set. Then for all Borel set $F\subset [-M,M]^{d}$, we have
 \begin{itemize}
     \item[i)]There exists a constant $c_1>0$ depending only on $a,b,M,\zeta$, and the law of $B$, such that 
     \begin{equation}\label{lower cap F}
         c_1\,\mathcal{C}_{d-\zeta}\left(F\right)\leq \mathbb{P}\left\{ B(E)\cap F\neq \emptyset \right\},
     \end{equation}
     where $\mathcal{C}_{\alpha}(\cdot)$ is the $\alpha$-Bessel-Riesz type capacity associated to the Euclidean metric on $\mathbb{R}^d$. 
     \item[ii)] If in addition to the Hypothesis \ref{H1}, the function $\gamma$ satisfies the condition \eqref{condition raisonable}, there exists a constant $c_2>0$ depending only on $a,b,M, \zeta$, and the law of $B$, such that 
     \begin{equation}\label{upper Haus F}
         \mathbb{P}\left\{B(E)\cap F\neq \emptyset\right\}\leq c_2\, \mathcal{H}^{d-\zeta}\left(F\right),
     \end{equation}
     where $\mathcal{H}^{\alpha}(\cdot)$ is defined as the Hausdorff measure of order $\alpha$ associated to the Euclidean metric on $\mathbb{R}^d$ when $\alpha>0$, and is supposed equal to one when $\alpha\leq 0$.
 \end{itemize}
 \end{proposition}
 \begin{proof}
 We start by proving $(i)$. By using \eqref{lower bnd hitting}, it suffice to show that 
 \begin{equation}\label{compare capacities}
     \mathcal{C}_{\rho_{\delta},d}(E\times F)\geq c_3\, \mathcal{C}_{d-\zeta}(F).
 \end{equation}
 Let us suppose that $\mathcal{C}_{d-\zeta}(F)>0$, otherwise there is nothing to prove. Let $0<\eta<\mathcal{C}_{d-\zeta}(F)$, then there exists a probability measure $\mu$ supported on $F$ such that 
 \begin{equation}
     \mathcal{E}_{d-\zeta}(\mu):=\dint\dint\varphi_{d-\zeta}\left(\|x-y\|\right)\mu(dx)\mu(dy)\leq \eta^{-1}.
 \end{equation}
Since the measure $\nu$ satisfies \eqref{zeta-regular set}, by applying Lemma \ref{lem estim krnl} in the metric space $([0,1],\delta)$ we get that for all $x,y\in F$ we have     
\begin{equation}\label{estim krnl}
    \dint_E\dint_E\frac{\nu(ds)\nu(dt)}{\left(\max\left\{\delta(s,t),\|x-y\|\right\}\right)^d}\leq c_4\,\varphi_{d-\zeta}\left( \|x-y\|\right).
\end{equation}
Now, since $\nu\otimes\mu$ is a probability measure on $E\times F$, by applying Fubini's theorem and the last estimation \eqref{estim krnl}, we obtain
\begin{align}
    \mathcal{E}_{\rho_{\delta},d}(\nu\otimes \mu)=\dint_{E\times F}\dint_{E\times F}\frac{\nu(ds)\mu(dx)\nu(dt)\mu(dy)}{\left(\max\left\{\delta(s,t),\|x-y\|\right\}\right)^d}\leq c_5\,\dint\dint\varphi_{d-\zeta}\left(\|x-y\|\right)\mu(dx)\mu(dy)\leq c_5\,\eta^{-1}.
\end{align}
Hence, we have $\mathcal{C}_{\rho_{\delta},d}\left(\nu\otimes\mu\right)\geq c_5^{-1}\, \eta$. By making $\eta\uparrow \mathcal{C}_{d-\zeta}(F)$ we get the desired inequality in \eqref{compare capacities}.

For proving $(ii)$, by using \eqref{upper bnd hitting}, it suffice to show that
\begin{equation}\label{compar Hausd meas}
    \mathcal{H}_{\rho_{\delta}}^{d}\left(E\times F\right)\leq c_6\, \mathcal{H}^{d-\zeta}\left(F\right).
\end{equation}
We assume that $\mathcal{H}^{d-\zeta}\left(F\right)<\infty$ and $d>\zeta$, otherwise there is nothing to prove. Let $\eta>\mathcal{H}^{d-\zeta}\left(F\right)$ be arbitrary, then there is a covering $\left(B(x_n,r_n)\right)_{n\geq 1}$ of $F$ such that 
\begin{align}\label{covering F}
F\subset \bigcup_{n=1}^{\infty}B(x_n,r_n) \quad \text{and}\quad \sum_{n=1}^{\infty}(2r_n)^{d-\zeta}\leq \eta.   
\end{align}
For all $n\geq 1$, let $N_{\delta}\left(E,r_n\right)$ be the smallest number of $\delta$-balls of radius $r_n$  required to cover $E$. Then  the family $\left\{B_{\delta}(t_{n,j},r_n)\times B(x_n,r_n): 1\leq j\leq N_{\delta}(E,r_n)\,\,, n\geq 1\right\}$ form a covering of $E\times F$ by open balls of radius $r_n$ for the metric $\rho_{\delta}$. Let $0<r<1$, we define the so called packing number $P_{\delta}\left(E,r\right)$ which is defined to be the greatest number of disjoint $\delta$-balls $B_{\delta}\left(t_j,r\right)$ centered in $t_j\in E$ with radius $r$. The lower part in \eqref{zeta-regular set} implies that 
\begin{align*}
    c_1\,P_{\delta}(E,\delta)\, r^{\zeta}\leq \sum_{j=1}^{P_{\delta}(E,r)}\nu\left(B_{\delta}(t_j,r)\right)\leq 1.
\end{align*}
Using the well known fact that $N_{\delta}\left(E,2r\right)\leq P_{\delta}\left(E,r\right)$, we obtain that 
\begin{align}\label{estim cover nbr}
    N_{\delta}\left(E,r\right)\leq c_7\, r^{-\zeta},
\end{align}
where the constant $c_7$ depends on $E$ only. Putting all those previous facts together, we have
\begin{align}
    \sum_{n=1}^{\infty}\sum_{j=1}^{N_{\delta}\left(E,r_n\right)}\left(2r_n\right)^{d}\leq c_8\,\sum_{n=1}^{\infty}\left(2r_n\right)^{d-\zeta}\leq c_8\, \eta. 
\end{align}
Then by making $\eta \downarrow \mathcal{H}^{d-\zeta}(F)$, the inequality in \eqref{compar Hausd meas} holds, which finishes the proof.
 \end{proof}

\begin{remark}
We note that when $E=[a,b]$, with $0<a<b<1$, lower and upper bounds have been obtained in Theorem 2.5 and Theorem 4.6 in \cite{Eulalia&Viens2013} in terms of the capacity  term $\mathcal{C}_{K}(F)$ and the Hausodrff measure term $\mathcal{H}_{\varphi}(F)$, respectively, and under some reasonable conditions on $\gamma$, where the kernel $K$ and the function $\varphi$ are defined in terms of $\gamma$. Those bounds could be regained again by just proving that 
\begin{align}
    \mathcal{C}_{\rho_{\delta},d}(E\times F)\geq c_9\, \mathcal{C}_{K}(F) \quad \text{ and  }\quad \mathcal{H}_{\rho_{\delta}}^{d}\left(E\times F\right)\leq c_{10}\, \mathcal{H}_{\varphi}\left(F\right). 
\end{align}
The proof holds from the same reasoning of Proposition \ref{bnds in trm of F}, and by using also some technics in the proof of Theorem 2.5 and Theorem 4.6 in \cite{Eulalia&Viens2013}. We do not give it here for some reasons of the length of paper. 
\end{remark}

\section{Hausdorff dimension of the random intersection $B(E)\cap F$ and $E\cap B^{-1}(F)$ }

Having known from the previous section that the random intersections $B(E)\cap F$ and $E\cap B^{-1}(F)$ are non-empty under some conditions on the Borel sets $E\subset (0,1)$ and $F\subset \mathbb{R}^d$, it is worth interesting to ask how large are those random intersections. Therefore, the natural way to know more information about their size is to calculate their Hausdorff dimension.
The following result may give an answer to this question.
\begin{theorem}\label{Haus dim intersect}
Assume again that the Hypothesis \ref{H1} holds. Then for all $0<a<b<\infty$ and $M>0$, let $E\subset [a,b]$ and $F\subset [-M,M]^d$ be two compact sets 
, then we have

\begin{itemize}
    \item[i)]
    \begin{itemize}
        \item[i-1)\label{lower 1}] For all $\eta>0$, $\mathbb{P}\left\{ \dim_{\delta}\left(E\cap B^{-1}(F)\right)\geq \dim_{\delta}(E)+\dim_{Euc}(F)-d-\eta\right\}>0$, and
        \item[i-2)] If the function $\gamma$ satisfies the condition \eqref{nice condition} then we have, a.s. 
        \begin{equation*}
            \dim_{\delta}\left(E\cap B^{-1}(F)\right)\leq \dim_{\rho_{\delta}}(E\times F)-d \label{upper 1}
        \end{equation*}
    \end{itemize}
        \item[ii)] If $\dim_{\delta}\left(E\right)\leq d$, we have 
    \begin{itemize}
        \item[ii-1)] For all $\eta>0$, $\mathbb{P}\left\{ \dim_{Euc}\left(B(E)\cap F\right)\geq \dim_{Euc}(F)+\dim_{\delta}(E)-d-\eta\right\}>0$, and
        \item[ii-2)] Again under \eqref{nice condition}  we have, a.s. \begin{equation*}
            \dim_{Euc}\left(B(E)\cap F\right)\leq \dim_{\rho_{\delta}}\left(E\times F\right)-d \label{upper 2}
        \end{equation*} 
    \end{itemize}
       \item[iii)] If $\dim_{\delta}(E)>d$, we have 
    \begin{itemize}
        \item[iii-1)] For all $\eta>0$, $\mathbb{P}\left\{ \dim_{Euc}\left(B(E)\cap F\right)\geq \dim_{Euc}(F)-\eta\right\}>0$, and 
        \item[iii-2)] $\dim_{Euc}\left(B(E)\cap F\right)\leq \dim_{Euc}(F)$ a.s.
    \end{itemize}
\end{itemize}
\end{theorem}

Let $Y:\Omega\rightarrow \mathbb{R}_+$ be a positive random variable, the essential supremum norm $\|Y\|_{L^{\infty}(\mathbb{P})}$ is defined by 
$$ \|Y\|_{L^{\infty}(\mathbb{P})}:=\sup\{\theta\geq 0  : \mathbb{P}\left(Y\geq \theta\right)>0 \}. $$
As an application of the previous theorem, we have the following corollary
    
\begin{corollary}\label{Cor dim H-scale}
If we assume that $B$ is a $d$-dimensional Gaussian process such 
that for each component $B_i$,  the commenturability condition \eqref{commensurate} holds with $\gamma(r)=r^{H}L(r)$, where $H\in (0,1)$ and $L(\cdot)$ is a slowly varying function. Then for any Borel sets $E\subset [a,b]$ and $F\subset [-M,M]^d$, we have
\begin{itemize}
    \item[i)] \begin{equation}\label{infinite norm 1}
        \dim_{Euc}(E)+H(\dim_{Euc}(F)-d)\leq \left\|\dim_{Euc}\left(B^{-1}(F)\cap E\right)\right\|_{L^{\infty}\left(\mathbb{P}\right)}\leq H\left(\dim_{\rho_{H}}(E\times F)-d\right),
    \end{equation}
    where $\dim_{\rho_H}\left(\cdot\right)$ is the Hausdorff dimension associated to the metric $\rho_H$  on $\mathbb{R}_+\times \mathbb{R}^d$, where  $\rho_{H}((s,x),(t,y)):=\max\{|t-s|^H,\|x-y\|\}$.
    \item[ii)] If $\dim_{Euc}(E)\leq H\,d$, we have 
    \begin{equation}\label{infinite norm 2}
        \frac{\dim_{Euc}(E)}{H}+\dim_{Euc}(F)-d\leq \left\|\dim_{Euc}\left(B(E)\cap F\right)\right\|_{L^{\infty}\left(\mathbb{P}\right)}\leq \dim_{\rho_{H}}(E\times F)-d
    \end{equation}
    \item[iii)]If $\dim_{Euc}(E)>H\,d$, we have 
    \begin{equation}\label{infinite norm 3}
        \left\|\dim_{Euc}\left(B(E)\cap F\right)\right\|_{L^{\infty}\left(\mathbb{P}\right)}=\dim_{Euc}(F).
    \end{equation}
\end{itemize}
\end{corollary}
\begin{remark}
We note that the equality between the upper and lower bounds in \eqref{infinite norm 1} and \eqref{infinite norm 2} occur when $\dim_{Euc}(E)=\operatorname{Dim}_{Euc}(E)$ or $\dim_{Euc}(F)=\operatorname{Dim}_{Euc}(F)$, because of the following comparison 
\begin{align}
    \begin{aligned}
    \frac{\dim_{Euc}(E)}{H}+ \dim_{Euc}(F)\leq \dim_{\rho_H}(E\times F)&\leq \min\left\{\frac{\operatorname{Dim}_{Euc}(E)}{H}+ \dim_{Euc}(F),\frac{\dim_{Euc}(E)}{H}+ \operatorname{Dim}_{Euc}(F)\right\}
    ,\label{estim parabolic dim 1}
    \end{aligned}
\end{align}
for any Borel sets $E\subseteq \mathbb{R}_+$, $F\subseteq \mathbb{R}^d$.
\end{remark}

\begin{proof}[proof of Corollary \ref{Cor dim H-scale}]
 For the lower bound it suffice to apply the previous theorem. Indeed, by Proposition \ref{prop RVF} the process $B$ satisfies the condition \eqref{condition raisonable}. Since $L(\cdot)$ is slowly varying function, we have 
 \begin{align}\label{slow variation}
     L(r)=o\left(r^{-\varepsilon}\right) \text{ near $0$ for any $\varepsilon>0$ small enough,}
 \end{align}
   then we can repeat the same argument used in $(i)$ of Remark \ref{examples of delta-dim} to show that $\dim_{\delta}(E)=\frac{\dim_{Euc}(E)}{H}$. 
 For the upper bound, using again the previous theorem, it suffice to check that $\dim_{\rho_{\delta}}(\cdot)\equiv\dim_{\rho_H}(\cdot)$. Indeed, thanks to the property \eqref{slow variation} again; it is easy to check, for any $\varepsilon>0$ small enough, that we have $\dim_{\rho_H}(G)\leq \dim_{\rho_{\delta}}(G)\leq \dim_{\rho_{H-\varepsilon}}(G)$ for all Borel set $G\subset \mathbb{R}_+\times \mathbb{R}^d$, then using Proposition 2.5 of Hakiki and Erraoui \cite{Erraoui&Hakiki2020} we obtain 
 \begin{equation*}
     0\leq \dim_{\rho_{\delta}}(G)-\dim_{\rho_H}(G)\leq \dim_{\rho_{H-\varepsilon}}(G)-\dim_{\rho_H}(G)\leq \frac{1}{H}-\frac{1}{H-\varepsilon},
 \end{equation*}
and by making $\varepsilon\downarrow 0$, the desired equality follows.
\end{proof}

Before proving both of Theorem \ref{Haus dim intersect}, we need the following lemma, which will be helpful to establish the lower bounds part.

\begin{lemma}\label{lem estim krnl}
Let $(X,\rho)$ be a bounded metric space, such that there exists a probability measure $\mu$ supported on $X$ which satisfies 
\begin{align}\label{Frost cond 2}
    \mu\left(B_{\rho}(u,r)\right) \leq C_1 r^{\kappa},
\end{align}
for all $u\in X$, $r>0$, where $C_1>0$ and $\kappa>0$ are two constants. Then for any $\theta>0$, there exists $C_2>0$ such that
\begin{align}
    \int_X\int_X\frac{\mu(du)\mu(dv)}{\left(\max\{\rho(u,v),r\}\right)^{\theta}}\leq C_2\, \varphi_{\theta-\kappa}(r),
\end{align}
for all $r\in (0,1)$.
\end{lemma}
\begin{proof}
 Since $\mu$ is a probability measure, it suffice to estimate the quantity $I:=\sup_{v \in X} \int_X \frac{\mu(du)}{\left(\max\{\rho(u,v),r\}\right)^{\theta}}$. When $\theta<\kappa$, we have 
 $$ I\leq \sup_{v\in X}\int_X\frac{\mu(du)}{\rho(u,v)^{\theta}}<\infty.$$
 When $\theta\geq \kappa$, We decompose this last integral into two parts $I_1$ and $I_2$, where
 \begin{align*}
     I_1=\int_{\{u: \rho(u,v)\leq r \}}\frac{\mu(du)}{r^{\theta}} \quad \text{ and }\quad I_2=\int_{\{u:\rho(u,v)\geq r\}}\frac{\mu(du)}{\rho(u,v)^{\theta}}.
 \end{align*}
 By using \eqref{Frost cond 2} we get 
 \begin{align}\label{I1 power}
     I_1\leq C_3 r^{\kappa-\theta}.
 \end{align}
 For estimating $I_2$, we assume that $diam(X)\leq 1$, and we set $k(r):=\inf\{k: 2^{-k}\leq r\}$. Then we have 
 \begin{align}
     \{u:\rho(u,v)\geq r\}\subset \bigcup_{k=1}^{k(r)}\{u:2^{-k}\leq \rho(u,v)<2^{-k+1}\}
 \end{align}
 Using again \eqref{Frost cond 2}
 \begin{align}\label{I2 geom}
     \begin{aligned}
         I_2&\leq \sum_{k=1}^{k(r)}2^{k\,\theta}\mu\left(\{u: 2^{-k}\leq \rho(u,v)<2^{-k+1}\}\right)\\
         &\leq \sum_{k=1}^{k(r)}2^{k(\theta-\kappa)}
     \end{aligned}
 \end{align}
 It follows from the definition of $k(r)$ that $2^{-k(r)}\leq r< 2^{-k(r)+1}$. Then, for $\theta=\kappa$ we get easily that 
 \begin{align}\label{I2 log}
     I_2\leq C_4\,\log(1/r),
 \end{align}
  for the case $\theta>\kappa$, we use a comparison with a geometric series to obtain \begin{align}\label{I2 power}
      I\leq C_5 r^{\kappa-\theta}.
  \end{align} 
  Putting \eqref{I1 power}, \eqref{I2 geom}, \eqref{I2 log}, and \eqref{I2 power} all together, we get the desired estimation.
\end{proof}
Let us prove now Theorem \ref{Haus dim intersect}
\begin{proof}[proof of Theorem \ref{Haus dim intersect}]First, we note that we can assume without loss of generality that the diameter of $E$ is smaller than $\varepsilon$. Let us prove $(i)$, we assume that $\dim_{Euc}(F)+\dim_{\delta}(E)>d$ otherwise there nothing to prove. We may assume also that $\dim_{\delta}(E)<\infty$\footnote{Because when $\dim_{\delta}(E)=\infty$, the result will take the form "$\mathbb{P}\{\dim_{\delta}(E\cap B^{-1}(F))\geq \eta\}>0$ for any $\eta>0$" and the proof follows from the same reasoning}. Let $0<\eta<\dim_{Euc}(F)+\dim_{\delta}(E)-d$, the definition of $\delta$-Hausdorff dimension ensures that there is $\nu \in \mathcal{P}(E)$ such that $\mathcal{E}_{\delta,\alpha}(\nu)<\infty$, where $\alpha:=\dim_{\delta}(E)-\eta/2$, and the classical Frostman's theorem ensures that there is $\mu\in \mathcal{P}(F)$ such that 
\begin{align}
    \mu(B(x,r))\leq C_1\, r^{\beta} \quad \text{ for all $x\in \mathbb{R}^d$ and $r>0$,}
\end{align}
where $\beta:=\dim_{Euc}(F)-\eta/2$, and $C_1$ is a positive constant. Let us consider the sequence of random measures $(\nu_n)_{n\geq 1}$ on $E$ defined as 
\begin{align}\label{seq of measure}
    \begin{aligned}
        \nu_n(ds)&=\left(\int_F (2 \pi n)^{d / 2} \exp \left(-\frac{n\|B(s)-x\|^{2}}{2}\right) \mu(d x)\right)\nu(ds)\\
        &=\left(\int_F\int_{\mathbb{R}^{d}} \exp \left(-\frac{\|\xi\|^{2}}{2 n}+i\langle\xi, B(s)-x\rangle\right) d \xi\,\mu(d x)\right)\nu(ds)
    \end{aligned}
\end{align}
Denote the total mass of $\nu_n$ by $\|\nu_n\|=\nu_n(E)$. We want to verify the following claim
\begin{equation}\label{claim 2}
    \mathbb{E}\left(\left\|\nu_{n}\right\|\right) \geqslant C_{2}, \quad  \quad \mathbb{E}\left(\left\|\nu_{n}\right\|^{2}\right) \leqslant C_{3}
\end{equation}
and 
\begin{equation}\label{claim energy}
    \mathbb{E}\left(\mathcal{E}_{\delta,\zeta}(\nu_n)\right)\leqslant C_4
\end{equation}
where $\zeta:=\beta+\alpha-d$, and the constants $C_2$, $C_3$, and $C_4$ are independent of $n$. For the first inequality in \eqref{claim 2} we use the same technique as in \eqref{claim 1}. We prove only the estimation \eqref{claim energy}, and  the second estimation in \eqref{claim 2} can be deduced from same technique of the last one. Indeed, we express the expected energy in \eqref{claim energy} as
\begin{align}\label{estimate energy}
    \begin{aligned}
\mathbb{E}\left( \int_E\int_E\frac{\nu_n(ds)\nu_n(dt)}{\delta(t,s)^{\zeta}} \right)=\int_{E^2}&\frac{\nu(ds)\nu(dt)}{\delta(t,s)^{\zeta}} \int_{F^2}\mu(dx)\mu(dy) \int_{\mathbb{R}^{2 d}} e^{-i(\langle\xi, x\rangle+\langle\eta, y\rangle)} \, \\
&\times \exp \left(-\frac{1}{2}(\xi, \eta)\left(n^{-1} I_{2 d}+\operatorname{Cov}(B(s), B(t))\right)(\xi, \eta)^{T}\right) d \xi \, d \eta\\
\leq C_5\, &\int_{E^2}\frac{\nu(ds)\nu(dt)}{\delta(t,s)^{\zeta}} \int_{F^2}\frac{\mu(dx)\mu(dy) }{\left(\max \left\{\gamma(|t-s|),\|x-y\|\right\}\right)^{d}}\\
\\
\leq C_6\,& \int_{E}\int_E\frac{\nu(ds)\nu(dt)}{\delta(t,s)^{\zeta+d-\beta}}
=C_6\, \mathcal{E}_{\delta,\alpha}(\nu)<\infty,
\end{aligned}
\end{align}
where the first inequality follows from Lemma \ref{lem lower bnd}, and the second inequality follows by applying Lemma \ref{lem estim krnl} to $X=F$ and $\rho$ is the euclidean metric.

Now, it is clear that \eqref{claim 2} combined with the Paley Zygmund inequallity (c.f. Kahane \cite{Kahane}, p.8) ensure that $\{\nu_n: n\geq 1\}$ has a subsequence that converge weakly to a finite measure $\nu_{ \infty}$ supported on $E\,\cap B^{-1}(F)$ which is positive on an event of positive probability (larger than $C_2^2/2C_3$). We use in \eqref{claim energy} Fatou's lemma and the lower semicontinuity of the energy $\mathcal{E}_{\delta,\zeta}(\cdot): \nu\mapsto \int\int \delta(t,s)^{-\zeta}\nu(ds)\nu(dt)$ on the space of positive measures $\mathcal{M}^+([0,1])$ equipped with the weak topology (see for example \cite{Landkof}, pg. 78), we can deduce that $\mathcal{E}_{\delta,\alpha}(\nu_{\infty})<\infty$ a.s., and by definition of $\delta$-Hausdorff dimension we deduce that $\mathbb{P}\left\{\dim_{\delta}\left(E\,\cap B^{-1}(F)\right)\geq \zeta \right\}\geq C_2^2/2C_3$. This finishes the proof of the lower bound in (i-1).

For the upper bound in (i-2), let us fix an arbitrary $\zeta>\dim_{\rho_{\delta}}(E\times F)-d$, and $0<\varepsilon<\zeta-\dim_{\rho_{\delta}}(E\times F)+d$, then $\mathcal{H}_{\rho_{\delta}}^{d+\zeta-\varepsilon}(E\times F)=0$. Let $\eta>0$ small enough, the definition of the $\rho_{\delta}$-Hausdorff measure ensures that there is a covering of $E\times F$ by balls $\{B_{\rho_{\delta}}((t_i,x_i),r_i), i\geq 1\}$ in $\left(\mathbb{R}_+\times\mathbb{R}^d,\rho_{\delta}\right)$ with small radii $r_i$, such that 
\begin{equation}\label{cover product 2}
    E\times F\subseteq \bigcup_{i=1}^{\infty}B_{\rho_{\delta}}((t_i,x_i),r_i)\quad  \text{with }\quad  \sum_{i=1}^{\infty}(2r_i)^{d+\zeta-\varepsilon}\leq \eta.
\end{equation}
Since for any $i\geq 1$, the ball $B_{\rho_{\delta}}((t_i,x_i),r_i)$ is nothing but the Cartesian product of $B_{\delta}(t_i,r_i)$ and $B(x_i,r_i)$, it is not difficult to check that
$$ E\cap B^{-1}(F)\subseteq \bigcup_{\left\{\, i:\, B_{\delta}(t_i,r_i)\cap   \textbf{B}^{-1}\left(B(x_i,r_i)\right)\neq \varnothing\right\}} B_{\delta}(t_i,r_i).$$
Hence, we have 
\begin{equation}
    \begin{aligned}\label{estim expect measure 1}
        \mathbb{E}\left(\mathcal{H}_{\delta}^{\zeta}\left(E\, \cap B^{-1}(F)\right)\right)&\leq \mathbb{E}\left(\sum_{i=1}^{\infty}(2r_i)^{\zeta}\operatorname{1}_{\{\, B_{\delta}(t_i,r_i)\cap   \textbf{B}^{-1}\left(B(x_i,r_i)\right)\neq \varnothing \}}\right)\\
        &\leq \sum_{i=1}^{\infty}(2r_i)^{\zeta}\mathbb{P}\left\{\, B_{\delta}(t_i,r_i)\cap   \textbf{B}^{-1}\left(B(x_i,r_i)\right)\neq \varnothing \right\}\\
        &\leq C_1 \sum_{i=1}^{\infty}(2r_i)^{d+\zeta-\varepsilon}\leq C_1 \eta,
    \end{aligned}
\end{equation}
where the last inequality follows from injecting the condition \eqref{nice condition} within \eqref{estim small ball 1} for $\varepsilon^{\prime}=\varepsilon/d$. Since $\eta>0$ is arbitrary, we conclude that $\mathcal{H}_{\delta}^{\zeta}\left(E\,\cap B^{-1}(F)\right)=0$ a.s., and then $\dim_{\delta}\left(E\,\cap B^{-1}(F)\right)\leq \zeta$. By making $\zeta \downarrow \dim_{\rho_{\delta}}(E\times F)-d$, we get the desired inequality.

For $(ii)$ we proceed by the same method, but just changing the roles between $E$ and $F$. For the lower bound part, let again $0<\eta<\dim_{Euc}(F)+\dim_{\delta}(E)-d$ small enough, then there is $\mu\in \mathcal{P}(F)$ such that $\mathcal{E}_{\beta}(\mu)<\infty$, where $\beta:=\dim_{Euc}(F)-\eta/2$, and there is $\nu\in \mathcal{P}(E)$ such that 
\begin{equation}\label{Frost cond 3}
    \nu(B_{\delta}(t,r))\leq C_2\, r^{\alpha} \quad \text{ for all $t\in (0,1)$ and $r>0$,}
\end{equation}
where $\alpha=\dim_{\delta}(E)-\eta/2$. We consider in this case $(\mu_n)_{n\geq 1}$ to be a sequence of random measures on $F$ defined as 
\begin{align*}
    \begin{aligned}
        \mu_n(dx)=\left(\int_F (2 \pi n)^{d / 2} \exp \left(-\frac{n\|B(s)-x\|^{2}}{2}\right) \nu(d s)\right)\mu(dx)
        \end{aligned}
\end{align*}
The estimation of $\mathbb{E}\left(\|\mu_n\|\right)$ from bellow is easy to check. We just estimate the expectation of the energy $\mathbb{E}\left(\mathcal{E}_{\zeta}(\mu_n)\right)$, and the estimation of $\mathbb{E}\left(\|\mu_n\|^2\right)$ follows from the same lines also. Indeed, for $\zeta=\alpha+\beta-d$, we have 
\begin{align}
    \begin{aligned}
\mathbb{E}\left( \int_F\int_F\frac{\mu_n(dx)\mu_n(dy)}{\|x-y\|^{\zeta}} \right)&
\leq C_3\,\int_{F^2}\frac{\mu(dx)\mu(dy)}{\|x-y\|^{\zeta}} \int_{E^2}\frac{\nu(ds)\nu(dt) }{\left(\max \left\{\gamma(|t-s|),\|x-y\|\right\}\right)^{d}}\\
\\
&\leq C_4\, \int_{F}\int_F\frac{\mu(dx)\mu(dy)}{\|x-y\|^{\zeta+d-\alpha}}
=C_4\, \mathcal{E}_{\beta}(\mu)<\infty,
\end{aligned}
\end{align}
where the second inequality come from an application of Lemma \ref{lem estim krnl} to $X=E$ and $\rho:=\delta$ ( we note that $\theta:=d$ and $\kappa:=\alpha=\dim_{\delta}(E)-\eta/2$, and $\theta>\kappa$ because of the condition "$\dim_{\delta}(E)\leq d$"). Repeating the same argument as above, we deduce that $\mathbb{P}\left\{\dim_{Euc}\left( B(E)\cap F\right)\geq \zeta \right\}> 0$. Which finishes the proof of the lower bound in (ii-1). For the upper bound in (ii-2), we repeat the same covering techniques used above in \eqref{cover product 2} and \eqref{estim expect measure 1}, it suffice to remark that in this case, the random set $B(E)\cap F$ is covered by the family of balls $\left\{ B(x_i,r_i)\, \text{ s.t. } \, B_{\delta}(t_i,r_i)\cap   \textbf{B}^{-1}\left(B(x_i,r_i)\right)\neq \varnothing \right\}$. 

For $(iii)$, the upper bound is trivial (by monotonicity of the Hausdorff dimension). The lower bound can be deduced from the same argument in $(ii)$, it suffice to take $\zeta=\dim_{Euc}(F)-\eta$ , and the condition $\dim_{\delta}(E)>d$ ensure that, the measure $\nu$ can be chosen such that \eqref{Frost cond 3} is satisfied for some $d<\alpha<\dim_{\delta}(E)$.
\end{proof}

\bibliographystyle{abbrv}

\begin{thebibliography}{9}
\small
\bibitem{Adler 2} R. J. Adler. \textit{An introduction to continuity, extrema, and related topics for general Gaussian processes}, IMS Lecture Notes-Monograph Series, Volume 12, (1990).
 
\bibitem{Bingham et al} N. H. Bingham.  C. M. Goldie and  J. L. Teugels. \textit{Regular Variation} Cambridge University Press
 Cambridge, (1987).
 
 \bibitem{BiPe}	C. J. Bishop and Y. Peres. \textit{Fractals in probability and analysis.} Cambridge Studies in Advanced Mathematics, 162. Cambridge University Press, Cambridge, (2017).

\bibitem{Chen&Xiao} Z. Chen and Y. Xiao. On intersections of independent anisotropic Gaussian random fields. Sci. China Math. 55, no. 11, 2217-2232, (2012).

\bibitem{CD} J. Cuzick and J. P. DuPreez.  Joint continuity of Gaussian local times. Ann. Probab. 10, no. 3, 810-817, (1982).
 
\bibitem{Kahane} J. P. Kahane. \textit{Some random series of functions}, Second edition. Cambridge Studies in Advanced Mathematics, 5. Cambridge University Press, Cambridge, (1985).

\bibitem{Hutchinson} Hutchinson, J.E.: Fractals and self similarity, Indiana Univ. Math. J., 30, 713-747, (1981). 

\bibitem{Tay&Wats}S. J. Taylor and N. A. Watson. A Hausdorff measure classification of polar sets for the heat equation. Math. Proc. Cambridge Philos. Soc. 97, no. 2, 325--344, (1985).

\bibitem{Eulalia&Viens2013} E. Nualart and F. Viens. Hitting probabilities for general Gaussian processes. (2013), available at https://arxiv.org/abs/1305.1758.

\bibitem{Hawkes}J. Hawkes. Local Properties of Some Gaussian Processes. Z. Wahrscheinlichkeitstheorie und
Verw. Gebiete, 40, 309–315, (1997).



\bibitem{Falconer} K. J. Falconer. \textit{Fractal geometry-mathematical foundations and applications}. John Wiley Sons, Chichester, (1990).

\bibitem{Ho95} J.D. Howroyd. On dimension and on the existence of sets of finite positive Hausdorff measure. Proc. Lond. Math. Soc., III. 70, 581–604 (1995).

\bibitem{Peres&Morters} P. Mörters and Y. Peres. \textit{Brownian motion.} Cambridge Series in Statistical and Probabilistic Mathematics, 30. Cambridge University Press, Cambridge, (2010).

\bibitem{Bierme&Xiao} H. Biermé  and C. Lacaux  and Y. Xiao. Hitting probabilities and the Hausdorff dimension of the inverse images of anisotropic Gaussian random fields. Bull London Math. Soc. 41, 253-273, (2009).

\bibitem{Marta&Calleja2021} M. sanz-solé and A. H. Calleja. Anisotropic Gaussian random fields: criteria for hitting probabilities and applications. J. Stoch. PDE Anal. Comp. 42 (2021).


\bibitem{Erraoui&Hakiki 3} M. Erraoui, and Y. Hakiki. Fractional Brownian motion with deterministic drift: Hausdorff dimension of the intersection of the image with a non-random Borel set. Work in progress, (2021).

\bibitem{Erraoui&Hakiki2020} M. Erraoui and Y. Hakiki. Images of fractional Brownian motion with deterministic drift: Positive Lebesgue measure and non-empty interior. Submitted paper, available at https://arxiv.org/abs/2112.02055, (2020). 

\bibitem{Xiao97} Y. Xiao. Hölder conditions for the local times and the Hausdorff measure of the level sets of Gaussian random fields. Probab. Theory Relat. Fields 109, 129–157, (1997).



\bibitem{Geman&Horowitz} D. Geman and J. Horowitz. Occupation densities. Ann. Probab. \textbf{8}, 1-67, (1980).


\bibitem{Landkof} L. S. Landkof. \textit{Foundations of modern potential theory}, Springer-Verlag, Berlin (1972).

\bibitem{Viens&Mocioalca2005} O. Mocioalca, and F. Viens. Skorohod integration and stochastic calculus beyond the fractional Brownian scale; J. Functional Analysis 222, 385-434, (2005).


\bibitem{KX2015} D. Khoshnevisan and Y. Xiao. Brownian motion and thermal capacity. Ann. Probab. 43, no. 1, 405-434, (2015).

\bibitem{Khosh} D. Khoshnevisan.\textit{  Multiparameter processes. An introduction to random fields}.  Springer-Verlag, New York, (2002). 

\bibitem{Ouahhabi&Tudor2013} H. Ouahhabi and C.A. Tudor. Additive functionals of the solution to fractional stochastic  heat equation. J. of Fourier Analysis and Application 19, 777-791, (2013).

\bibitem{Seneta} E. Seneta. Regularly varying functions, Lecture Notes in Math.508, Springer-Verlag, (1976).

\bibitem{Talagrand} M. Talagrand. Regularity of gaussian processes. Acta Math. 159, 1–2, 99–149, (1987).




\bibitem{SX}R. Shieh. and Y. Xiao. Images of Gaussian random fields: Salem sets and interior points. Studia Math. 176, no. 1, 37-60, (2006). 

\bibitem{Xiao2009} Y. Xiao. Sample path properties of anisotropic Gaussian random fields. A Minicourse on Stochastic Partial Differential Equations, (D Khoshnevisan and F Rassoul-Agha, editors), Lecture Notes in Math, 1962: 145-212. New York: Springer, (2009)

\end{thebibliography}

\end{document}